\documentclass[11pt,xcolor=dvipsnames,svgnames,table,reqno]{amsart}

\usepackage[left=1.5in, right=1.5in, bottom=1.25in,
height=8in]{geometry}

\input xy
\xyoption{all}
\usepackage{accents}
\usepackage{epsfig}
\usepackage{xcolor}
\usepackage{amsthm}
\usepackage{amssymb}
\usepackage{bbm}
\usepackage{amsmath}
\usepackage{amscd}
\usepackage{amsopn}
\usepackage{graphicx}

\usepackage{tikz}
\usetikzlibrary{shapes,shadows,arrows}
\usepackage{tikz-cd}

\usepackage{xspace}

\usepackage{hhline}
\usepackage{easybmat}
\usepackage{caption}   
\usepackage{relsize}

\usepackage{url}
\usepackage{enumitem, hyperref}\hypersetup{colorlinks}

\usepackage{stmaryrd}

\usepackage[T1]{fontenc}

\colorlet{purpleB70}{blue!70!red}

\colorlet{orangeR65}{red!65!yellow}

\definecolor{red2}{HTML}{d41173}

\definecolor{neongreen}{HTML}{1bf702}

\definecolor{radicalred}{HTML}{FF355E}

\definecolor{denim}{HTML}{1560BD}

\definecolor{darkcyan}{rgb}{0.0, 0.55, 0.55}

\definecolor{cilek}{HTML}{FF43A4}

\definecolor{mor}{HTML}{9F00C5}

\definecolor{phlox}{rgb}{0.87, 0.0, 1.0}

\definecolor{fluorescentpink}{HTML}{FF1493}

\definecolor{napiergreen}{rgb}{0.16, 0.5, 0.0}

\definecolor{kellygreen}{rgb}{0.3, 0.73, 0.09}

\definecolor{parisgreen}{HTML}{ 50C878 }

\definecolor{palatinateblue}{rgb}{0.15, 0.23, 0.89}

\definecolor{ceruleanblue}{rgb}{0.16, 0.32, 0.75}

\definecolor{brandeisblue}{rgb}{0.0, 0.44, 1.0}

\definecolor{KLMblue}{HTML}{0FC0FC}

\definecolor{cinnamon}{rgb}{0.82, 0.41, 0.12}

\definecolor{darkorange}{rgb}{1.0, 0.55, 0.0}

\definecolor{darktangerine}{rgb}{1.0, 0.66, 0.07}

\definecolor{deepcarrotorange}{rgb}{0.91, 0.41, 0.17}

\definecolor{internationalorange}{HTML}{FF4F00}

\definecolor{persimmon}{HTML}{EC5800}

\definecolor{pumpkin}{HTML}{FF7518}

\definecolor{darkred}{rgb}{1,0,0} 
\definecolor{darkgreen}{rgb}{0,0.7,0}
\definecolor{darkblue}{rgb}{0,0,1}

\hypersetup{colorlinks,
linkcolor=darkblue,
filecolor=darkgreen,
urlcolor=darkred,
citecolor=darkgreen}

\makeatletter
\def\reflb#1#2{\begingroup
    #2%
    \def\@currentlabel{#2}%
    \phantomsection\label{#1}\endgroup
}
\makeatother

\numberwithin{equation}{section}
\newtheorem{Theorem}{Theorem}
\numberwithin{Theorem}{section}

\newtheorem   {Lemma}[Theorem]{Lemma}

\newtheorem   {Proposition}[Theorem]{Proposition}
\newtheorem   {Corollary}[Theorem]{Corollary}
\theoremstyle {definition}

\theoremstyle {remark}
\newtheorem   {Remark}[Theorem]{Remark}
\newtheorem   {Example}[Theorem]{Example}

\def    \eps    {\epsilon}

\newcommand{\CH}{{\mathcal H}}
\newcommand{\CA}{{\mathcal A}}
\newcommand{\CC}{{\mathcal C}}

\newcommand{\CU}{{\mathcal U}}
\newcommand{\CV}{{\mathcal V}}

\newcommand{\CS}{{\mathcal S}}

\newcommand{\supp}{\operatorname{supp}}

\newcommand{\id}{{\mathit id}}

\newcommand{\const}{{\mathit const}}

\newcommand{\fa}{{\mathfrak a}}
\newcommand{\fb}{{\mathfrak b}}

\newcommand{\fh}{{\mathfrak h}}

\newcommand{\tOmega}{\tilde{\Omega}}

\newcommand{\tth}{\tilde{h}}

\newcommand{\tz}{\tilde{z}}

\newcommand{\tf}{\tilde{f}}

\newcommand{\tPsi}{\tilde{\Psi}}

\newcommand{\talpha}{\tilde{\alpha}}

\newcommand{\tH}{\tilde{H}}
\newcommand{\tG}{\tilde{G}}

\newcommand{\CB}{{\mathcal B}}

\newcommand{\CW}{{\mathcal W}}

\def    \F      {{\mathbb F}}

\def    \C      {{\mathbb C}}
\def    \R      {{\mathbb R}}

\def    \Z      {{\mathbb Z}}
\def    \N      {{\mathbb N}}
\def    \Q      {{\mathbb Q}}

\def    \T      {{\mathbb T}}
\def    \CP     {{\mathbb C}{\mathbb P}}

\def    \12     {{\frac{1}{2}}}

\def    \p      {\partial}

\def    \SH     {\operatorname{SH}}

\def    \SO     {\operatorname{SO}}

\def    \HF     {\operatorname{HF}}

\def    \vol     {\operatorname{vol}}

\newcommand    \htop  {\operatorname{h_{\scriptscriptstyle{top}}}}

\newcommand   \slope {\operatorname{\mathit{slope}}}
\newcommand   \rmax {r_{\max}}

\newcommand   \WW {\widehat{W}}

\newcommand \Quad {\R^n_{{\geq 0}}}

\newcommand \Tr  {\operatorname{Tr}}
\newcommand \dSBM {\operatorname{d_{\scriptscriptstyle{SBM}}}}

\begin{document}


\setlength{\smallskipamount}{6pt}
\setlength{\medskipamount}{10pt}
\setlength{\bigskipamount}{16pt}





\title [Barcode Growth for Toric-Integrable Systems]{Barcode Growth
  for Toric-Integrable Hamiltonian Systems}

\author[Erol Barut]{Erol Barut}
\author[Viktor Ginzburg]{Viktor L. Ginzburg}

\address{EB and VG: Department of Mathematics, UC Santa Cruz, Santa
  Cruz, CA 95064, USA} \email{ebarut@ucsc.edu} \email{ginzburg@ucsc.edu}

\subjclass[2020]{53D40, 37B40, 37J12, 37J55}

\keywords{Periodic orbits, Reeb flows, Floer homology, persistence modules, toric domains,
  symplectic toric manifolds}

\date{\today} 

\thanks{This work is partially supported by NSF grant DMS-2304206 and
  Simons Foundation grant MP-TSM-00002529}

\begin{abstract}
  We continue investigating the connection between the dynamics of a
  Hamiltonian system and the barcode growth of the associated Floer or
  symplectic homology persistence module, focusing now on completely
  integrable systems. We show that for convex/concave or real analytic
  toric domains and convex/concave or real analytic completely
  integrable Hamiltonians on closed toric manifolds the barcode has
  polynomial growth with degree (i.e., slow barcode entropy) not
  exceeding half of the dimension. This slow polynomial growth
  contrasts with exponential growth for many systems with sufficiently
  non-trivial dynamics. We also touch upon the barcode growth function
  as an invariant of the interior of the domain and use it to
  distinguish some open domains.
\end{abstract}

\maketitle



\tableofcontents

\section{Introduction}
\label{sec:intro}
In this paper we continue investigating the connection between the
dynamics of a Hamiltonian system and the barcode growth of the
associated Floer or symplectic homology persistence module. We show
that the barcode has at most polynomial growth when the system is
completely integrable in a very strong sense. Namely, for
convex/concave or real analytic toric domains and convex/concave or
real analytic completely integrable Hamiltonians on closed symplectic
toric manifolds the barcode has polynomial growth with degree (aka
slow barcode entropy) not exceeding half of the dimension. This
phenomenon contrasts with that the exponential growth rate (aka
barcode entropy) is connected to the topological entropy of the
system; see
\cite{CGG:Entropy,CGG:Growth,CGGM:Reeb,FeLS,Fe1,Fe2,GGM,Me}. In
dimensions, two or three the barcode entropy is equal to the
topological entropy. Thus the growth of the barcode reflects both
trivial (e.g., completely integrable with toric singularities) and
non-trivial dynamics of the system. This is the main result of the
paper.

Applications of persistent homology methods to symplectic dynamics was
pioneered in \cite{PS}; see also, e.g., \cite{PRSZ,UZ} for further
references. However, this aspect of the connection between the
symplectic/Floer homology barcode and dynamics beyond periodic is even
a more recently discovered phenomenon and it is not clear at the
moment how far this connection goes.  The filtered fixed point Floer
homology and symplectic homology cannot detect all dynamics
features. Indeed, Hamiltonian and Reeb pseudo-rotations have the same
structure of periodic orbits as ordinary ``rotations'' and hence the
same filtered homology, but can have very different (e.g., ergodic)
dynamics; see \cite{AK,Ci,CGG:Mult,CS,GG:PR,GG:PRvsR,JS,LRS,Ka:Izv}
for relevant results and constructions. Yet such examples are
extremely non-generic and it is hypothetically possible that
generically all or most of the dynamics is detectable, at least in
lower dimensions, by Floer theory even when it is interpreted in such
a narrow sense.

A related question going back to \cite{CFHW} which we touch upon here
is to what extent the Reeb dynamics on the boundary of a Liouville
domain is determined by the interior of the domain. For instance, it
is a standard fact, proved in various versions several times, that the
filtered symplectic homology is an invariant of the interior up to
exact symplectomorphisms; see, e.g., \cite{CFHW,Gu} and Proposition
\ref{prop:Inv1}. Thus, as a consequence of the results in
\cite{CGGM:Reeb}, in dimension four the topological entropy of the
Reeb flow on the boundary is an invariant of the interior up to exact
symplectomorphisms of the interior. To put the same idea more
broadly, symplectic dynamics invariants can be used to distinguish
open domains.

Furthermore, motivated by the results and constructions in
\cite{LRSV,Hu,Us-SBM}, we extend the definition of the barcode growth
function, essentially by continuity, to closed star-shaped domains in
$\R^{2n}$. (Note that one can also define the filtered symplectic
homology in this case, in an invariant way, but it is not clear in
what sense this homology would be a persistence module and how to
define its barcode; see Remark \ref{rmk:SH-C0}.) We do not know if the
growth function is an invariant of the interior, but it allows us to
define barcode entropy and slow barcode entropy which are
invariants. The key tool in the proof of invariance is the notion of
symplectic Banach--Mazur distance; see \cite{PRSZ,SZ-SBM,Us-SBM}. We
extend our upper bounds on the barcode growth function to
convex/concave toric domains in $\R^{2n}$ with non-smooth boundary. In
particular, we show that the slow entropy is bounded from above by
$n$. The Reeb flow on the boundary of a $C^\infty$-generic smooth
star-shaped domain in $\R^4$ has positive topological entropy; see
\cite{CDHR} and also \cite{CKMS}. Hence, as a byproduct of our main
results, we see that the interior of such a domain is not
symplectomorphic to the interior of a convex/concave, possibly
non-smooth, domain. The same is true in all dimensions when the Reeb
flow has a hyperbolic invariant set with positive topological entropy;
see \cite{CGGM:Reeb}. These theorems refine and generalize some of the
results in \cite{Hu} where a different dynamics invariant is
introduced to distinguish open domains.

The paper is organized as follows. In Section \ref{sec:results} we
briefly review relevant definitions, and state and discuss our main
results in further detail. Then the upper bound on the barcode growth
for convex/concave or real analytic completely integrable Hamiltonians
on closed toric symplectic manifolds (Theorem \ref{thm:toric-sympl})
is proved in Section \ref{sec:pf-toric-sympl}. The proof is quite
similar to the argument for toric domains but requires less machinery
and is technically simpler. In Section \ref{sec:prelim} we discuss
filtered symplectic homology and its properties from the perspective
of persistence modules. This material is quite standard, except the
case of non-smooth domains, although its treatment varies between
different sources. We closely follow \cite{CGGM:Reeb}.  We also touch
upon the question of invariance in Section \ref{sec:Inv} and this part
is mainly new. Finally, in Section \ref{sec:main-prfs} we prove the
upper bounds (Theorems \ref{thm:toric-Reeb} and
\ref{thm:toric-nonsmooth}) for the barcode growth for convex/concave,
not necessarily smooth, and real analytic toric domains. In the smooth
convex/concave case the argument shares some common ingredients with
the proofs in \cite{GH}.

\medskip\noindent{\bf Acknowledgements.} The authors are grateful to
Erman \c Cineli, Ba\c sak G\"urel, Jean Gutt, Leonid Polterovich,
Michael Usher and Jun Zhang for useful discussions and remarks.

\section{Main results}
\label{sec:results}

\subsection{Toric domains}
\label{sec:toric}
The definition of a convex/concave toric domain, while conceptually
well-established, varies subtly between different sources and is
sometimes difficult to track down specifically; see Remark
\ref{rmk:def-convex-concave}. In this section, we spell out the
definition suitable for our purposes and used throughout the paper.

Consider the standard action of the torus $\T^n$ on $\C^n=\R^{2n}$ by
rotations of the complex coordinate axes and let
$$
\mu\colon \C^n\to \Quad, 
\quad (z_1,\ldots,z_n)\mapsto
\pi\big(|z_1|^2,\ldots,|z_n|^2\big)
$$
be its moment map, where $\R^n_{\geq 0}\subset\R^n$ is the positive
quadrant.  We denote the intersection of the open positive quadrant
$\R^n_{>0}$ with the unit sphere $S^{n-1}\subset \R^n$ by $\Delta_+$
and by $\bar{\Delta}_+=\R^n_{\geq 0}\cap S^{n-1}$ its closure. Let
$f\colon \bar{\Delta}_+\to (0,\infty)$ be a continuous function. Note
that $f$ extends to a positive continuous function on a neighborhood
of $\bar{\Delta}_+$ in $S^{n-1}$. Let $(r,\theta)$, where
$r\in [0,\infty)$ and $\theta\in S^{n-1}$, be the standard polar
coordinates on $\R^n$. Set
\begin{equation}
  \label{eq:Omega}
\Omega_f:=\{(r,\theta)\in\R^n\mid r\leq f(\theta), \,\theta\in
\bar{\Delta}_+\}\subset \R^n_{\geq 0}
\end{equation}
and
$$
\p\Omega_f:=\{(r,\theta)\mid r= f(\theta), \,\theta\in
\bar{\Delta}_+\}.
$$
Clearly, $\Omega_f$ is a star-shaped domain in the positive quadrant. By
definition, a \emph{toric} domain in $\R^{2n}$ is the
$\T^{2n}$-invariant closed star-shaped domain
$$
W_f=\mu^{-1}(\Omega_f)\subset \R^{2n}.
$$
Sometimes, we will use the notation $\Omega$ and $W$ or $W_\Omega$
when the function $f$ is clear from the context.

We say that $W_f$ is \emph{smooth} when $f$ is smooth, i.e., $f$
extends to a smooth function on a neighborhood of $\bar{\Delta}_+$ in
$S^{n-1}$. Then
$$
\p W_f=\mu^{-1}(\p \Omega_f)
$$
is a smooth hypersurface (with corners) in $\R^{2n}$ and the lines
through the origin are never tangent to $\p W_f$.  Likewise, $W_f$ is
\emph{real analytic} if $f$ admits a real analytic extension to a
small neighborhood of $\bar{\Delta}_+$ in $S^{n-1}$. Then $\p W_f$ is
a real analytic hypersurface in $\R^{2n}$.

The Reeb flow on $\p W_f$, to the extent it is defined (e.g., when
$W_f$ is smooth), is \emph{completely integrable} in a very strong
sense with first integrals given by the coordinates $x_i$ on $\R^n$
or, to be more precise, by the pull-backs $x_i\circ\mu$. (Complete
integrability is understood here as that the functions $x_i\circ\pi$
are in involution, invariant under the flow and independent almost
everywhere on $\p W_f$; see \cite[Sec.\ 49]{Ar}.) Moreover, by
construction, these integrals generate a Hamiltonian
$\T^n$-action. Hence, an addition, the set where these functions fail
to be independent has a rather simple structure. To be more precise,
in our case the integrals have toric singularities.

Somewhat misleadingly, we say that $W_f$ is \emph{convex} if for some
extension of $f$ (still denoted by $f$) to a small neighborhood $U$ of
$\bar{\Delta}_+$ in $S^{n-1}$ the set
\begin{equation}
  \label{eq:convex-ext}
\{(r,\theta)\mid r\leq f(\theta), \,\theta\in
U\}\subset \R^n
\end{equation}
is convex. Note that then the closed domains $\Omega_f$ and $W_f$ are
automatically convex, but our convexity requirement is somewhat
stronger than just that $\Omega_f$ and $W_f$ are convex; see Example
\ref{ex:toric-not-smooth}. We say that $W_f$ is
\emph{concave} if for some extension of $f$ to a small neighborhood
$U$ of $\bar{\Delta}_+$ in $S^{n-1}$ the set
\begin{equation}
  \label{eq:concave-ext}
\{(r,\theta)\mid r\geq f(\theta), \,\theta\in
U\}\subset \R^n
\end{equation}
is convex. (Note that the inequalities in \eqref{eq:convex-ext} and
\eqref{eq:concave-ext} go in the opposite directions. In this case, we
also say that the sets $\Omega_f$ and \eqref{eq:convex-ext} are
concave.)  Then the set $\R^n\setminus \Omega_f$ is automatically
convex, but again our requirement is more restrictive.  The domain
$W_f$ is said to be \emph{smooth convex/concave} if in these
definitions the extension of $f$ can be taken smooth.

\begin{Remark}
  \label{rmk:def-convex-concave}
  It is worth keeping in mind that, as we have pointed out, there are
  in circulation several slightly different definitions of
  convex/concave domains.  For instance, our definition of convex
  toric domains is more general than the ones from, say, \cite{GH} and
  \cite{Hu:GT16} for $n=2$ and slightly less general than the
  definition in \cite{CG} for $n=2$. (When $n=2$, it would become
  equivalent to the latter if we simply required $\Omega_f$ to be
  convex and dropped the condition that $f$ extends to a neighborhood
  $U$ of $\bar{\Delta}_+$ so that the extended domain,
  \eqref{eq:convex-ext}, is still convex.) Our definition of concave
  domains is slightly more restrictive than the ones in
  \cite{CG,GH,Hu:GT16} due to again the extension condition.
\end{Remark}

\begin{Example}
  \label{ex:toric-not-smooth}
  Let $D\subset \R^2_{\geq 0}$ be a closed disk tightly rolled inside
  the positive quadrant, i.e., such that the boundary circle $\p D$ is
  tangent to both coordinate axes. The tangency points devide the
  circle $\p D$ into two arcs: the smaller arc $\Gamma_-$ which is
  closer to the origin and the larger arc $\Gamma_+$ which is farther
  away. The arc $\Gamma_+$ bounds a convex subset $\Omega_+$ in
  $\R^2_{\geq 0}$ and $W_+=\mu^{-1}(\Omega_+)$ is a toric
  domain. However, according to our definitions $W_+$ is not smooth.
  Nor is it convex even though it is a convex subset of $\R^{2n}$. In
  a similar vein, the complement to the region $\Omega_-$ bounded by
  $\Gamma_-$ in the positive quadrant is convex. However, the toric
  domain $W_-=\mu^{-1}(\Omega_-)$ is neither smooth nor concave
  according to our definition.
\end{Example}

\subsection{Barcode growth for toric-integrable Reeb flows}
\label{sec:results-Reeb}
Let $(W,\alpha)$ be a Liouville domain.  Fix a ground field $\F$.
Usually, we suppress $\F$ and $\alpha$ in the notation. The
(non-equivariant) filtered symplectic homology spaces $\SH^s(W)$ over
$\F$ form a persistence module $\SH(W)$ parametrized by the action
$s\in\R$. Referring the reader to Section \ref{sec:prelim} for
details, here we only point out that while the definition of
$\SH^s(W)$ is quite standard when $s$ is not in the action spectrum
$\CS(\alpha)$; the case of $s\in\CS(\alpha)$ requires some
attention. The grading of the symplectic homology is inessential for
our purposes and we treat $\SH^s(W)$ as an ungraded vector space over
$\F$.

Denote by $\fb_\eps(W,s)$ the number of bars of length greater than
$\eps>0$ in the barcode of this persistence module beginning in the
range $[0,\, s)$; see Section \ref{sec:persistence-def}. This
invariant of persistence modules of geometric nature was first
considered in \cite{CSEHM} and then used repeatedly as a reflection of
the geometrical and dynamics features of the underlying problem. (See,
e.g., \cite{BHPW, BP3S2} and references therein in addition to the
works on barcode entropy cited in the introduction.)  We will write
$\fb_\eps(s)$ when $(W,\alpha)$ is clear from the context and refer to
$\fb_\eps(s)$ as the \emph{barcode (growth) function}. This is an
increasing function in $s$ and $1/\eps$, locally constant as a
function of $s$ in the complement to $\CS(\alpha)\cup \{0\}$.

Here we are interested in the growth rate of $\fb_\eps(s)$ when the
Reeb flow of $\alpha$ is toric-integrable. Namely, in this section we
will assume that $W=W_f$ is a smooth toric domain. Then, as we have
already mentioned, the Reeb flow on $\p W$ is defined and completely
integrable in a very strong sense.  Under suitable technical
additional conditions on $\p \Omega$, this translates to strong
restrictions on $\SH(W)$ and the growth of $\fb_\eps(s)$. Namely, we
have the following.

\begin{Theorem}
  \label{thm:toric-Reeb}
Assume that $W_\Omega$ is real analytic, or smooth and convex or concave.
Then
\begin{equation}
  \label{eq:toric-Reeb}
\fb_\eps(s)\leq C_n(\Omega) s^n + C_0(\Omega)
\end{equation}
for all $s$, where the constants $C_n(\Omega)$ and $C_0(\Omega)$
are independent of $\eps>0$ and, of course, $s$.
\end{Theorem}

The theorem is proved in Section \ref{sec:pf-toric-Reeb} by, roughly
speaking, bounding from above the number of generators in the
underlying Floer complex. Furthermore, in many situations the constants
$C_n(\Omega)$ and $C_0(\Omega)$ can be determined quite explicitly.

This constraint on the behavior of $\fb_\eps(s)$ also provides an
upper bound on the growth of another invariant, measuring the growth
of the filtered homology or to some extent the concentration of
bars. Namely, set
$$
\fh(s)=\sup_{s'<s}\dim \SH^{s'}(W)\leq \infty.
$$
Clearly,
$$
\fh(s)\leq \sup_{\eps>0} \fb_\eps(s).
$$
Hence, since the constants $C_n(\Omega)$ and $C_0(\Omega)$ in
\eqref{eq:toric-Reeb} are independent of $\eps$, we arrive at the
following.
\begin{Corollary}
  \label{cor:toric-Reeb}
In the setting of Theorem \ref{thm:toric-Reeb}, 
$$
\fh(s)\leq C_n(\Omega) s^n + C_0(\Omega) 
$$
for all $s$, where the constant $C_n(\Omega)$ and $C_0(\Omega)$ are
independent of $s$.
\end{Corollary}

It is useful to contrast Theorem \ref{eq:toric-Reeb} with the results
on the growth of $\fb_\eps(s)$ for some Reeb flows $\varphi^t_\alpha$
with ``interesting'' dynamics. Assume that $M=\p W$ is smooth and set
$\log^+:=\max\{0,\log\}$ where the logarithm is taken base two. Then,
the \emph{$\eps$-barcode entropy} of $W$ and the \emph{barcode entropy} of $W$ are
defined as
\begin{equation}
  \label{eq:hbar}
  \hbar_\eps(W):=\limsup_{s\to\infty}\frac{\log^+\fb_\eps(s)}{s} \textrm{ and } \hbar(W):=\lim_{\eps\to
    0+}\hbar_\eps(W).
\end{equation}
We have
\begin{equation}
  \label{eq:hbar<htop}
  \hbar(W)\leq \htop(\varphi_\alpha),
\end{equation}
where $\htop$ stands for the topological entropy; see \cite{FeLS} and
also \cite{CGGM:Reeb}. In particular, $\hbar(W)<\infty$. Furthermore,
as is proved in \cite[Thm.\ B]{CGGM:Reeb}, $\fb_\eps(s)$ grows
exponentially fast whenever the flow has a closed hyperbolic invariant
set $K$ with positive topological entropy $\htop(K)$. To be more
precise,
\begin{equation}
  \label{eq:hbar>htop}
\hbar_\eps(W)\geq
\htop(\varphi_\alpha|_K)
\end{equation}
for all sufficiently small $\eps>0$. Furthermore, when $\dim M=3$, we have
\begin{equation}
  \label{eq:hbar=htop}
  \hbar(W)
=\htop(\varphi_\alpha) 
\end{equation}
by \cite[Cor.\ C]{CGGM:Reeb}. Identity \eqref{eq:hbar=htop} is a
consequence of \eqref{eq:hbar>htop}, \eqref{eq:hbar<htop} and the
results from \cite{LY,LS} showing that in dimension three all of
topological entropy comes from hyperbolic invariant sets and thus
extending the work of Katok, \cite{Ka}, for diffeomorphisms of
surfaces to flows. In particular, $\fb_\eps(s)$ grows exponentially
fast for all small $\eps>0$ when $\htop(\varphi_\alpha)>0$ and
$\dim M=3$.

The polynomial growth rate
\begin{equation}
  \label{eq:hbar-slow}
\hbar^{\scriptstyle{Slow}}(W):=\lim_{\eps\to 0+}\limsup_{s\to
  \infty}\frac{\log^+\fb_\eps(s)}{\log s}
\end{equation}
should be interpreted as the \emph{slow barcode entropy} of
$\varphi_\alpha$; cf., e.g., \cite{FrLS,KT}. Thus, in the setting of
Theorem \ref{eq:toric-Reeb}, $\hbar^{\scriptstyle{Slow}}(W)\leq n$ while
$\hbar^{\scriptstyle{Slow}}(W)=\infty$ when \cite[Thm.\ B,
Cor.\ C]{CGGM:Reeb} apply. It is not hard to see that in the latter case we
have $\fh(s)$ growing at least exponentially, with a lower bound
$\sup_\eps\fb_\eps(s)/s$ which may grow super-exponentially.

As is shown in \cite{CDHR} (see also \cite{CKMS})), a
$C^\infty$-generic Reeb flow in dimension three has positive
topological entropy. This is an extension of similar results from
\cite{KLCN,LCS} for Hamiltonian diffeomorphisms of surfaces to Reeb
flows. Therefore, in dimension three, $\hbar>0$ and $\fb_\eps(s)$ for
small $\eps>0$ grows exponentially fast $C^\infty$-generically by
\cite[Cor.\ C]{CGGM:Reeb}.  We will return to Theorem
\ref{thm:toric-Reeb} and relevant questions and conjectures in
Section~\ref{sec:results-discuss}.

\subsection{Barcode growth for non-smooth convex/concave domains}
\label{sec:results-nonsmooth}
Throughout this section we will assume that $W$ is a bounded closed
star-shaped domain in $\R^{2n}$ with continuous but not necessarily
smooth boundary. In other words, by definition, $W$ is given by the
condition $r\leq f(\theta)$ where $f\colon S^{2n-1}\to (0,\infty)$ is
continuous but not necessarily smooth.

One could extend the definition of the filtered symplectic homology of
$W$ to this setting by approximating $W$ by star-shaped domains with
smooth boundary but it is unclear if this homology would form a
persistence module in the sense of the definition adopted here (
Section \ref{sec:persistence}) and how to define its barcode; see
Remark \ref{rmk:SH-C0}. Instead, it is sufficient for our purposes to
have the function $\fb_\eps(s)$ defined essentially by continuity; see
Section \ref{sec:C0} and, in particular, \eqref{eq:fb-C0}. Once the
growth function $\fb_\eps(s)$ is defined, the definitions of barcode
entropy and slow barcode entropy carry over word-for-word.

Assume next that $W$ is a convex/concave toric domain as defined in
Section \ref{sec:toric}. Abusing notation, we will write $f$ for both
$f\colon \bar{\Delta}_+\to (0,\infty)$ and the pull-back
$f\circ\mu$. The function $f$ need not be smooth. However, we do
require $f$ to extend to a neighborhood $U$ of $\bar{\Delta}_+$ such
that the domain \eqref{eq:convex-ext} is convex/concave. (Note that
this extension of $f$ is automatically continuous on perhaps a smaller
neighborhood of $\bar{\Delta}_+$ due to the concavity/convexity
condition; see, e.g., \cite{Ro}.) Then we have the following analogue
of Theorem \ref{thm:toric-Reeb}, proved in Section
\ref{sec:pf-toric-nonsmooth}:

\begin{Theorem}
  \label{thm:toric-nonsmooth}
Assume that $W$ is convex or concave.
Then for every $\eps>0$ there exist constants $C_n(\Omega)$ and
$C_0(\Omega)$ independent of $s$ and $\eps>0$ such that
\begin{equation}
  \label{eq:toric-nonsmooth}
\fb_\eps(s)\leq C_n(\Omega) s^n + C_0(\Omega).
\end{equation}
\end{Theorem}

As a consequence of Theorem \ref{thm:toric-nonsmooth},
$\hbar^{\scriptstyle{Slow}}(W)\leq n$ and $\hbar(W)=0$ when $W$ is
toric convex or concave.

\subsection{Barcode growth for toric symplectic manifolds}
\label{sec:results-sympl}
Throughout this section, $(W^{2n},\omega)$ is a closed symplectic
manifold which for the sake of simplicity we assume to be weakly
monotone. Then for a Hamiltonian diffeomorphism
$\varphi=\varphi_H\colon W\to W$ generated by a Hamiltonian
$H\colon W\times S^1\to \R$ its filtered Floer homology $\HF^a(H)$ is
defined. As in the previous section the homology is taken over any
ground field $\F$, fixed in advance. If we attempted to drop the
condition that $W$ is weakly monotone, we would set $\F=\Q$. In
general, these homology spaces do not quite form a persistence module,
but the barcode of $\HF^s(H)$ is still defined (see \cite{UZ}), and by
slightly abusing notation we denote by $\fb_\eps(\varphi)$ the
number of bars in this barcode of length greater than $\eps>0$ as in
\cite{CGG:Entropy}. We are interested in the behavior of
$\fb_\eps\big(\varphi^k\big)$ as $k\to\infty$. When $\varphi$ is clear
from the context, we will write $\fb_\eps(k)$ for
$\fb_\eps\big(\varphi^k\big)$.

Next, assume in addition that $W^{2n}$ is toric, i.e., $W$ admits a
faithful Hamiltonian $\T^n$-action; see, e.g., \cite[Chap.\
XI]{CdS}. Let
$$
\mu\colon W\to \R^n
$$
be the moment map and $\Omega:=\mu(W)$ the moment polytope, in the
symplectic context, usually referred to as the Delzant polytope. (This
notation is based on a rather superficial analogy with the setting of
Section \ref{sec:results-Reeb}. In fact, $\bar{\Delta}_+$ is a much
better analogue of the moment polytope as is also clear from the proof
of Theorem \ref{thm:toric-sympl}.) The polytope $\Omega$
completely determines $(W, \omega)$ together with the action (see
\cite{De}), and as in Section \ref{sec:results-Reeb} we will use the
notation $W:=W_\Omega$ when $W$ is a closed toric manifold.

Let now $h\colon \Omega\to \R$ be a smooth function. Setting
$H=h\circ \mu$ we obtain an (autonomous) Hamiltonian on
$W_\Omega$. As before, the Hamiltonian flow generated by $H$ is completely
integrable in a very strong sense with $n$ first integrals
$x_i\circ \mu$ generating a Hamiltonian $\T^n$-action commuting with
the flow. Our next result is an analogue of Theorem
\ref{thm:toric-Reeb} for closed symplectic manifolds.

\begin{Theorem}
  \label{thm:toric-sympl}
Assume that $h$ is convex or concave or real
analytic. Then
\begin{equation}
  \label{eq:toric-sympl}
\fb_\eps\big(\varphi^k_H\big)\leq C_n(h) k^n + C_0(h)
\end{equation}
for all $k\in \N$, where the constants $C_n(h)$ and $C_0(h)$ are independent of
$\eps>0$ and, of course, $k$.
\end{Theorem}

Note that here $h$ is said to be real analytic whenever it extends to
a real analytic function on a neighborhood of $\Omega\subset\R^n$.
Theorem \ref{thm:toric-sympl} is proved in Section
\ref{sec:pf-toric-sympl}.  It is again illuminating to contrast the
upper bound \eqref{eq:toric-sympl} with the behavior of
$\fb_\eps\big(\varphi^k\big)$ for Hamiltonian diffeomorphisms with
``interesting'' dynamics.

By \cite[Thm.\ B]{CGG:Entropy}, $\fb_\eps\big(\varphi^k\big)$ grows
exponentially whenever $\varphi$ has a closed hyperbolic invariant set
$K$ with positive topological entropy $\htop(K)$. More precisely,

$$
\hbar_\eps(\varphi):=\limsup_{k\to\infty}\frac{\log^+\fb_\eps\big(\varphi^k\big)}{k}\geq \htop(K)
$$
for all sufficiently small $\eps>0$. Furthermore, when $\dim W=2$, we have
$$
\hbar(\varphi):=\lim_{\eps\to
  0+}\hbar_\eps(\varphi)
=\htop(\varphi) 
$$
by \cite[Thm.\ C]{CGGM:Reeb}. Ultimately, this is as a consequence of
\cite[Thm.\ B]{CGG:Entropy} and the results from \cite{Ka} mentioned
above.  Moreover, $C^\infty$-generically $\htop(\varphi)>0$ in
dimension 2 (see \cite{KLCN,LCS}), and hence
$\fb_\eps\big(\varphi^k\big)$ grows exponentially
$C^\infty$-generically for small $\eps>0$ when $\dim W=2$.

The \emph{slow barcode entropy} of
$\varphi$ is by definition the polynomial growth rate
$$
\hbar^{\scriptstyle{Slow}}(\varphi):=\lim_{\eps\to
  0+}\frac{\log^+\fb_\eps\big(\varphi^k\big)} {\log k}.
$$
Hence, in the setting of Theorem \ref{eq:toric-sympl},
$\hbar^{\scriptstyle{Slow}}(\varphi)\leq n$ while
$\hbar^{\scriptstyle{Slow}}(\varphi)=\infty$ when \cite[Thm.'s B and
C]{CGG:Entropy} apply, e.g., $C^\infty$-generically in dimension two.
We will continue this discussion in Section
\ref{sec:results-discuss}.

\begin{Remark}
  In the spirit of Section \ref{sec:results-nonsmooth}, it should be
  possible to extend the definition of the barcode or at least the
  barcode growth function to a suitable closure of the space of
  Hamiltonian diffeomorphisms and then prove an analogue of Theorem
  \ref{sec:results-sympl}. For aspherical manifolds and surfaces this
  has been done in \cite{BHS,LRSV,KS} with the $C^0$-closure.
  However, toric manifolds are never aspherical and a stronger notion
  of convergence would likely be needed; cf.\ \cite{OM} for instance.
  Moreover, it is not clear to us what the geometrical or dynamics
  significance of the upper bound, \eqref{eq:toric-sympl}, would be in
  that case. 
\end{Remark}  

\subsection{Invariance of the barcode growth}
\label{sec:results-inv}
It is a standard fact that the filtered symplectic homology of a
Liouville domain is an invariant of exact symplectomorphisms of the
interior; see Section \ref{sec:Inv} for a discussion and, in
particular, Proposition \ref{prop:Inv1} and, e.g., \cite{CFHW,Gu,Hu}
for relevant results. In other words, $\SH^s(W)=\SH^s(W')$ for any two
Liouville domains, e.g., smooth star-shaped domains in $\R^{2n}$, such
that the interiors $\mathring{W}$ and $\mathring{W}'$ are (exact)
symplectomorphic. As a consequence, $\fb_\eps(W,s)=\fb_\eps(W',s)$.

We do not know if the growth function is also an invariant of an
interior for non-smooth star-shaped domains in $\R^{2n}$. However, we
show that
$\hbar^{\scriptstyle{Slow}}(W)=\hbar^{\scriptstyle{Slow}}(W')$ and
$\hbar(W)=\hbar(W')$ for any two not necessarily smooth star-shaped
domains $W$ and $W'$ in $\R^{2n}$ such that their interiors
$\mathring{W}$ and $\mathring{W}'$ are symplectomorphic; see Corollary
\ref{cor:inv-nonsmooth}. (Since the domains are contractible all
symplectomorphisms are automatically exact.) Moreover, if one of the
domains is smooth, the domains have the same barcode growth function
by Theorem \ref{thm:inv-nonsmooth}: $\fb_\eps(W,s)=\fb_\eps(W',s)$.

Combining these results with the discussion in Section
\ref{sec:results-Reeb}, we arrive at the following corollary refining
and generalizing a result from \cite{Hu} obtained by a different
method.

\begin{Corollary}
  \label{cor:non-toric}
  Assume that $W^{2n}$ is a star-shaped domain with smooth boundary
  such that the Reeb flow $\varphi^t$ on $\p W$ has a hyperbolic
  invariant set with positive topological entropy, e.g., $\dim \p W=3$
  and $\htop(\varphi)>0$. Then the interior $\mathring{W}$ is not
  symplectomorphic to $\mathring{W}_\Omega$ when $W_\Omega$ is a
  convex or concave toric domain (not necessarily smooth) or
  $\p\Omega$ is real analytic. In particular, the interior of a
  $C^\infty$-generic star-shaped domain in $\R^4$ is not
  symplectomorphic to $\mathring{W}_\Omega$.
\end{Corollary}

\begin{Remark}
  Note that a similar statement for closed star-shaped domains with
  smooth boundary (in a more general and stronger form) readily
  follows from the non-trivial but standard fact that generic
  autonomous Hamiltonian systems are not completely integrable; see
  \cite{MM:book,MM:paper} and also \cite{BFRV} and references
  therein. Furthermore, we expect the corollary to hold, as stated,
  for all smooth Liouville domains $W$, not necessarily star-shaped in
  $\R^{2n}$. The proof, however, would require some (possibly, non-trivial)
  extension of the notion of symplectic Banach--Mazur distance.
\end{Remark}

When distinguishing Liouville domains with smooth boundary, invariance
of symplectic homology (Proposition \ref{prop:Inv1}) provides a host
of numerical invariants such as barcode entropy or slow barcode
entropy. By \cite[Cor.\ C]{CGGM:Reeb}, in dimension four this
translates into invariance of topological entropy:

\begin{Corollary}
  \label{cor:htop}
Let $W$ and $W'$ be 4-dimensional Liouville domains with smooth
boundary such that the interiors $\mathring{W}$ and $\mathring{W}'$
are exact symplectomorphic. The Reeb flows on $\p W$ and $\p W'$
have the same topological entropy.
\end{Corollary}

\subsection{Discussion:  conjectures and open questions}
\label{sec:results-discuss}
There are several possible overlapping directions of generalizing
Theorems \ref{thm:toric-Reeb} and \ref{thm:toric-sympl}. The first one
concerns relaxing the assumptions on the function $h$ or the boundary
$\p \Omega$. We see no reason why the upper bounds
\eqref{eq:toric-Reeb} and \eqref{eq:toric-sympl} would hold literally
as stated with $C_n$ and $C_0$ independent of $\eps$ when $h$ is just
a smooth function or $\Omega$ is just a toric star-shaped domain with
smooth boundary. However, we do expect $\fb_\eps$ to grow at most
subexponentially in this case or perhaps even slower.

In essence, the main reason for a slow growth of $\fb_\eps$ in Theorems
\ref{thm:toric-Reeb} and \ref{thm:toric-sympl} is that the underlying
Hamiltonian flow is completely integrable, i.e., there are $n$
functions in involution (first integrals) preserved by the flow and
independent on a sufficiently large, at least open and dense, set; see
\cite[Sec.\ 49]{Ar}. Hence, another direction is relaxing the
condition that the first integrals generate a $\T^n$-action, i.e.,
that $W$ is toric, and turning to more general completely integrable
systems. (Note also that once this condition is dropped the
convexity/concavity requirements do not appear to be any longer
meaningful.) We conjecture that $\fb_\eps$ grows at most
subexponentially or perhaps even slower than that when the flow and
the first integrals are real analytic or the integrals meet suitable
non-degeneracy conditions; see Remark \ref{rmk:vol} and, e.g.,
\cite{BO,El,MZ,Zu} and references therein.

Some restrictions on the first integrals are necessary. Indeed, an
example of a closed Riemannian 3-manifold $Q$ with a completely
$C^\infty$-integrable geodesic flow and an exponentially growing
$\pi_1(Q)$ was constructed in \cite{BT}. Then the geodesic flow
necessarily has positive topological entropy; \cite{Di}. Moreover,
then the set of conjugacy classes in $\pi_1(Q)$ also grows
exponentially, and hence so does the number of infinite bars
$\fb_\infty(s)$. Hence, $\fb_\eps(s)\geq \fb_\infty(s)$ grows
exponentially for every $\eps>0$.  In other words,
$\hbar_\eps\geq \hbar_\infty>0$. While at the moment we do not have a
reference or a construction, we believe that in a similar vein a
completely integrable Hamiltonian flow or diffeomorphism on a closed
symplectic manifold can have positive barcode entropy and therefore,
by \cite[Thm.\ A]{CGG:Entropy}, positive topological entropy. It would
also be interesting to have an example of a completely integrable Reeb
flow in dimension three with positive topological entropy. The same
question stands for Reeb flows on $S^{2n-1\geq 3}$ for barcode entropy
or at least topological entropy.

The results relating barcode growth and topological entropy have been
actually established in a more general setting than discussed in
Sections \ref{sec:results-Reeb} and \ref{sec:results-sympl}. Namely,
for a Hamiltonian diffeomorphism $\varphi$ one can relate the barcode
growth of the filtered Lagrangian Floer homology
$\HF\big(\varphi^k(L), L'\big)$ for a pair of Lagrangian submanifolds
$L$ and $L'$ (aka the relative barcode entropy) to the topological
entropy of $\varphi$; see \cite{CGG:Entropy,Me}. An analogue of this
relation between relative barcode entropy and topological entropy for
geodesic flows is investigated in \cite{GGM} and in \cite{Fe1,Fe2} for
Liouville domains by using wrapped Floer homology. In particular, in
all these cases the relative barcode entropy is bounded from above by
the topological entropy. Thus in the settings of Theorems
\ref{thm:toric-Reeb} and \ref{thm:toric-sympl} one can expect
subexponential growth of the barcode function for any pair of
Lagrangian or Legendrian submanifolds. (While we are not aware of any
published results to this account, we firmly believe that the
topological entropy is zero in the setting of these theorems.)
However, are there robust polynomial growth upper bounds as in these
theorems? For strategically chosen submanifolds the question should be
accessible by reasoning similar to the proofs of those
theorems. However, for a general pair of Lagrangian or Legendrian
submanifolds the situation could be much more involved.

Finally, note that while Theorems \ref{thm:toric-Reeb} and
\ref{thm:toric-sympl} provide an upper bound on the growth rate of
$\fb_\eps$, it is not clear what the actual behavior of $\fb_\eps$ is
in, say, the convex/concave case. Depending on $W$, the barcode
function $\fb_\eps$ can grow slower than these upper bounds. For
instance, it is not hard to see that when $\p W$ is an ellipsoid in
the setting of Section \ref{sec:results-Reeb}, i.e., $\Omega$ is cut
out from $\R^{n}_{\geq 0}$ by a hyperplane, $\fb_\eps(s)$ grows
linearly regardless of the dimension. In particular,
$\hbar^{\scriptstyle{Slow}}(W)=1$. When $W$ is a closed toric manifold
and $h$ is linear, $\fb_\eps(\varphi^k)$ is bounded and thus
$\hbar^{\scriptstyle{Slow}}(\varphi)=0$. (However, in both of these
examples the system is not strictly convex.)  For convex/concave toric
domains, the situation might be simpler than for symplectic toric
manifolds, but the answer is still unknown. For strictly
convex/concave domains we expect the upper bound from Theorem
\ref{thm:toric-sympl} to be sharp, i.e., $\fb_\eps(s)$ to grow as a
polynomial of degree $n$. A related question is that of the behavior
of the shortest bar $\beta_{\min}$: For instance, does $\beta_{\min}$
remain bounded away from 0 as $s\to\infty$ or $k\to\infty$ when
$\p\Omega$ or $h$ is strictly convex? Some numerical evidence, albeit
indirect, supporting the conjecture that this may indeed be the case
has been recently obtained by Pazit Haim-Kislev; \cite{HK}. However,
even in the setting of Theorem \ref{thm:toric-sympl} for $n=1$, i.e.,
when $h$ is a convex/concave function of the height on
$\CP^1=S^2\subset\R^3$, both questions are quite non-trivial because
of the affect of recapping.

\begin{Remark}[Volume growth]
  \label{rmk:vol}
  There is a different and slightly more general approach to the
  proofs of Theorems \ref{thm:toric-Reeb} and
  \ref{thm:toric-sympl}. Namely, the growth of $\fb_\eps(s)$ or
  $\fb_\eps(\varphi^k)$ is bounded from above by the growth of the
  volume of the graph of $\varphi^s_\alpha$ or $\varphi^k_H$; see
  \cite{CGG:Entropy, CGG:Growth, FeLS}. For completely integrable
  systems with toric singularities, it should not be hard to show that
  the volume grows polynomially with an upper bound as in the right
  hand sides of \eqref{eq:toric-Reeb} and \eqref{eq:toric-sympl}. The
  advantage of the approach we have chosen here is that it gives a
  ``hands-on'' upper bound on the number of generators and also more
  readily extends to the non-smooth setting of Section
  \ref{sec:results-nonsmooth}.
\end{Remark}  

\begin{Remark}[Equivariant barcode growth]
  \label{rmk:equiv}
  The upper bounds from Theorems \ref{thm:toric-Reeb},
  \ref{thm:toric-nonsmooth} and \ref{thm:toric-sympl} have analogues
  for the barcode growth of the $S^1$-equivariant symplectic or Floer
  homology over $\Q$. The proofs of such upper bounds are slightly
  more direct than and essentially contained in the proofs of these
  theorems due to the description of the underlying chain complexes
  in, say, \cite[Sec.\ 2.5]{GG:LS}; see also \cite{GH} for relevant
  calculations and constructions. Our choice of non-equivariant Floer
  and symplectic homology is dictated by the fact that the machinery
  of barcode entropy, used here for comparison purposes, is up to date
  systematically developed only in the non-equivariant setting. For
  instance, while \eqref{eq:hbar>htop} continues to hold for
  equivariant homology due to \cite[Thm.\ 3.1]{BP3S2} as is pointed
  out in \cite[Rmk.\ 2.4]{GGM}, it is not known if the opposite
  inequality, \eqref{eq:hbar<htop}, holds.
\end{Remark}

\section{Proof of Theorem \ref{thm:toric-sympl}}
\label{sec:pf-toric-sympl}
It is illuminating to start with the proof of Theorem
\ref{thm:toric-sympl}, for the argument requires little background and
sets the stage for the proofs of its contact analogues which are more
technical. We refer the reader to, e.g., \cite{CdS} for a concise and
self-contained introduction to symplectic toric manifolds.  We will
prove a result stronger then Theorem \ref{thm:toric-sympl}, giving an
upper bound on the number of generators of the Floer complex (after a
small perturbation) over the Novikov field, and hence on
$\fb_\eps\big(\varphi^k_H\big)$.

\begin{Theorem}
  \label{thm:toric-sympl2}
  Assume that $h$ is convex or concave or real analytic. Then, for a
  suitable arbitrarily $C^\infty$-small non-degenerate perturbation
  $\psi$ of $\varphi_H^k$, the number of fixed points of $\psi$ is
  bounded from above by $ C_n(h) k^n + C_0(h)$, where the constants
  $C_n(h)$ and $C_0(h)$ are independent of the perturbation and $k$.
\end{Theorem}

We have included the constant $C_0(h)$ in the upper bound only to
emphasize the similarity with the Reeb case. Since $k\geq 1$, this
constant can be absorbed into $C_n(h)$.

\begin{Remark}
  It is not hard to see from the proof of Theorem
  \ref{thm:toric-sympl2} that when $h$ is strictly convex/concave the
  upper bound provided by the theorem is sharp in the following
  sense. Namely, the Hamiltonian $kH$ is Morse--Bott non-degenerate
  and the components of its fixed-point set are tori of dimensions
  between $0$ and $n$, and the number of such tori grows as a
  polynomial of degree $n$. (In fact, this is true already for
  $n$-dimensional tori.) Moreover, for this type of a lower bound it
  is enough to just have a region when the Hessian of $h$ is
  non-degenerate.
\end{Remark}

\begin{proof}
  We will focus first on the convex case. The proof for a concave
  function $h$ is identical.  Let $\tth$ be a strictly convex
  $C^\infty$-small perturbation of $h$. To be more precise, we require
  the restriction of $\tth$ to any open face $\Delta$, of any
  dimension $d$, of the moment polytope $\Omega$ to be strictly
  convex. Set $G_\Delta=\nabla (h|_\Delta)$ and
  $\tG_\Delta=\nabla (\tth|_\Delta)$. We can treat $G_\Delta$ and
  $\tG_\Delta$ as maps $\Delta\to V_\Delta$, where $V_\Delta$ is the
  $d$-dimensional linear space parallel to the affine space containing
  $\Delta$. This is (the dual of) the Lie algebra of a torus, a
  certain quotient of $\T^n$, and hence $V_\Delta$ carries a canonical
  lattice $\Z^d$ and a canonical measure. We will fix a basis in
  $\Z^d$ and thus, in particular, identify $V_\Delta$ with $\R^d$.

  Then $\tG_\Delta(\Delta)\subset V_\Delta$ is an open subset of $V_\Delta$
  and the map $\tG_\Delta$ is one-to-one due to strict
  convexity. Moreover, $\tG_\Delta(\Delta)$ is contained in a fixed in
  advance open neighborhood $U_\Delta$ of the closure of
  $G_\Delta(\Delta)$. Note that we can make $\vol_d(U_\Delta)$ 
  arbitrarily close to $\vol_d\big(G_\Delta(\Delta)\big)$, where $\vol_d$ stands for the
$d$-dimensional volume.

The inverse image $L_w:=\mu^{-1}(w)$, where $w\in\Delta$, is a torus
of dimension $d$ invariant under the flow of $\tH=\tth\circ \mu$ and
an orbit of the $\T^n$-action. This torus is comprised entirely of
periodic orbits $\varphi_{\tH}$ if and only if $\tG_\Delta(w)$ is a
rational vector in $V_\Delta$. We will call such tori
rational. (Otherwise, no point in $L_w$ is periodic.) Denote the
components of $\tG_\Delta(w)$ by $p_i/q_i$, where, of course, $p_i$
and $q_i$ are relatively prime.  The minimal period $q$ of any
$x\in L_w$ is the least common multiple of the denominators
$q_1,\ldots, q_d$. (Note that $q$ is independent of the choice of the
basis in $\Z^d$.)

For $q\leq k$ the number of rational points in
$\tG_\Delta(\Delta)\subset U_\Delta$ is bounded from above by the
number of such points in $U_\Delta$. The later number grows as
$$
\vol_d(U_\Delta) k^d+ O(k^{d-1}),
$$
and this function is completely determined by $U_\Delta$ and thus is
independent of $\tth$ as long as $\tth$ is close to $h$.  Combining
these upper bounds for all $\Delta$ we conclude that the total number
of rational tori $L_w$ with $q\leq k$ is bounded from above by
$ C'k^n + C''$, where the constants $C'$ and $C''$ can be taken
independent of $\tth$.

Furthermore, it is easy to see that all such tori $L_w$ are
Morse--Bott non-degenerate. This is a consequence of the fact that $w$
is a regular point of $\tG_\Delta$ due to again strict
convexity. Applying a $C^\infty$-small perturbation to
$\varphi_{\tH}^k$, we can split each $L_w$ into $2^d$
non-degenerate fixed points. Setting $C_n(h)=2^nC'$ and
$C_0(h)=2^nC''$ we obtain a perturbation $\psi$ of $\varphi_{\tH}^k$,
and hence of $\varphi_{H}^k$, with the number of fixed points bounded
from above by $ C_n(h) k^n + C_0(h)$ as required. In the concave case
the argument is similar.

Next, assume that $h$ is real analytic, i.e., by definition, $h$ extends
to a real analytic function on a neighborhood of $\Omega$. Note that
in the above argument the convexity of $\tth$ was used only to ensure
that $\tG|_\Delta$ is one-to-one, and hence $L_w$ is a single torus,
and that this map does not have rational critical values, and hence $L_w$ is
Morse--Bott non-degenerate. Keeping the notation from the convex case,
we see that it is enough to find an arbitrarily small perturbation
$\tth$ of $h$ such that
\begin{itemize}

\item no rational vector $v\in V_\Delta$ is a critical value of $\tG_\Delta$;  
  
\item for any $\Delta$ and any regular value $v\in V_\Delta$ the
  number of inverse images $\tG^{-1}_\Delta(v)$ is bounded from
  above by some constant $N$ completely determined by $h$.

  \end{itemize}
Once these requirements are met we can set $C_n(h)=N2^nC'$ and
$C_0(h)=N2^nC''$ where $C'$ and $C''$ are as above.

Let $B$ be a sufficiently small closed ball in $\R^n$ centered at the
origin. Set $h_\lambda(x):=h(x)+\left<\lambda,x\right>$ and
$G_{\Delta,\lambda}:=\nabla (h_\lambda|_\Delta)$, where
$\lambda\in B$. In particular, $h_0=h$, and $h_\lambda$ and
$G_{\Delta,\lambda}$ are small perturbations of $h$ and,
respectively, $G_\Delta$. Denote by $\lambda_\Delta$ the orthogonal
projection of $\lambda$ to $V_\Delta$. Then, clearly,
$$
G_{\Delta,\lambda}=G_\Delta+\lambda_\Delta.
$$
For a fixed $k$, for an open and dense and full measure set of
$\lambda$'s the set of critical points of $G_{\Delta,\lambda}$
contains no rational points with $q\leq k$ for all $\Delta$; for the
sets of all faces $\Delta$ and such rational points are finite.  As a
consequence, for a full measure second Baire category set of
$\lambda\in B$, the set of critical values of all $G_{\Delta,\lambda}$
contains no rational points. Therefore, for such parameters $\lambda$,
the rational tori $L_w:=\mu^{-1}(w)$ are Morse--Bott non-degenerate
for the flow of $\tH=\tth\circ \mu$ with $\tth=h_\lambda$. Hence, the
first requirement is satisfied for all such $\tth$.

To finish the proof it remains to establish the second
requirement. Recall that for a real analytic map between manifolds of
the same dimension and a fixed compact set, the number of the inverse
images in that set of a regular value is bounded from above by a
constant depending only on the map and the set. (This is a consequence
of a similar statement for complex analytic maps, where the upper
bound is the ``local degree''.) Note that $v$ is a regular value of
$\tG_\Delta=G_{\Delta,\lambda}$ if and only if $(\lambda,v)$ is a
regular value of the map
$(\lambda,w)\mapsto \big(\lambda, G_{\Delta,\lambda}(w)\big)$. 
Let us apply this principle to each such map.  Thus the number of the
inverse images $\tG^{-1}_\Delta(v)$ for a regular value $v$ and any
$\Delta$ is bounded from above by some constant $N$, independent of
$v$. By construction, $N$ is completely determined by $h$, which
establishes the second requirement and completes the proof of the
theorem.
\end{proof}

\section{Symplectic homology as a persistence module}
\label{sec:prelim}

\subsection{Persistence modules}
\label{sec:persistence}
Persistence modules play a central role in the statements of our main
results. In this section, closely following \cite{CGGM:Reeb}, we
define the class of persistence modules suitable for our goals and
briefly touch upon their properties. We refer the reader to
\cite{PRSZ} for a general introduction to persistence modules and
their applications to geometry and analysis, although the class of
modules they consider is somewhat more narrow than the one we deal
with here, and also to \cite{BV, CdSGO16} for some of the more general
results and further references.

\subsubsection{Basic definitions}
\label{sec:persistence-def}
Fix a field $\F$ which we will suppress in the notation. Recall that a
\emph{persistence module} $(\CV,\pi)$ is a family of vector spaces $\CV_s$
over $\F$ parametrized by $s\in \R$ together with a functorial family
$\pi$ of structure maps. These are linear maps
$\pi_{st}\colon \CV_s\to \CV_t$, where $s\leq t$ and functoriality is
understood as that $\pi_{sr}=\pi_{tr}\pi_{st}$ whenever
$s\leq t\leq r$ and $\pi_{ss}=\id$. In what follows we often suppress
$\pi$ in the notation and simply refer to $(\CV,\pi)$ as $\CV$.  In such a
general form the concept is not particularly useful and usually one
imposes additional conditions on the spaces $\CV_t$ and the structure
maps $\pi_{st}$. These conditions vary depending on the context. Below
we spell out the framework most suitable for our purposes.

Namely, we require that the following four conditions to be met:

\begin{itemize}

\item[\reflb{PM1}{{(PM1)}}] There exists a bounded from below, nowhere
  dense subset $\CS\subset \R$, called the \emph{spectrum} of $\CV$,
  such that the persistence module $\CV$ is \emph{locally constant}
  outside $\CS$, i.e., $\pi_{st}$ is an isomorphism when $s\leq t$
  are in the same connected component of $\R\setminus \CS$.

\item[\reflb{PM2}{\rm{(PM2)}}] The persistence module $\CV$ is
  \emph{q-tame}: $\pi_{st}$ has finite rank for all $s<t$.
 
\item[\reflb{PM3}{\rm{(PM3)}}] \emph{Left-semicontinuity}: For all
  $t\in\R$,
  \begin{equation}
    \label{eq:semi-cont}
    \CV_t=\varinjlim_{s<t} \CV_s.
  \end{equation}

\item[\reflb{PM4}{\rm{(PM4)}}] \emph{Lower bound}: $\CV_s=0$ when $s<s_0$
  for some $s_0\in\R$. (Throughout the paper we will assume that
  $s_0=0$.)

\end{itemize}

A few comments on this definition are in order. First, note that as a
consequence of \ref{PM1} and \ref{PM2}, $\CV_s$ is finite-dimensional
and \ref{PM3} is automatically satisfied when $s\not\in
\CS$. Furthermore, Condition \ref{PM1} does not uniquely determine the
spectrum $\CS$: it is simply a requirement that such a set
exists. While one could take the minimal set with the required
properties as the spectrum of $\CV$, which would make it unique and
determined by $\CV$, this choice is not necessarily the most natural
and we prefer to think of the choice of $\CS$ as a part of the
persistence module structure. For instance, in the setting of Example
\ref{ex:sublevels}, it is convenient to take the set of critical
values of $f$ as $\CS$, but the minimal (aka homological)
spectrum can be stricter smaller. The same phenomenon can happen for
Floer or symplectic homology persistence modules.

By Condition \ref{PM3}, a persistence module is completely determined
by its restriction to a dense subset $\Gamma\subset \R$. In other
words, once $\CV_s$ and $\pi_{st}$ are defined for $s$ and $t$ in
$\Gamma$, the persistence module can be extended to $\R$ by
\ref{PM3}. (However, in general the extension might fail to meet some
of the requirements \ref{PM1}--\ref{PM4}.) For instance, Floer or
symplectic homology persistence modules are naturally defined only
for $s\not\in\CS$, i.e., for $\Gamma=\R\setminus\CS$, and then the
definition is extended to all $s\in\R$ via \ref{PM3}. 

By \ref{PM4}, we can always assume that $s_0\leq \inf \CS$, i.e.,
$\CS$ is bounded from below. We emphasize, however, that $\CS$ is not
assumed to be bounded from above, and it is actually not in many
examples we are interested in. In what follows it will sometimes be
convenient to include $s=\infty$ by setting
$$
\CV_\infty=\varinjlim_{s\to\infty} \CV_s.
$$
Finally, in all examples we encounter here $\CS$ has zero measure and,
in fact, zero Hausdorff dimension. This fact is never directly used in the
paper, but we do use the requirement that $\CS$ is nowhere dense.

A basic example motivating requirements \ref{PM1}--\ref{PM4} is that
of the sublevel homology of a smooth function.

\begin{Example}[Homology of sublevels]
\label{ex:sublevels}
Let $M$ be a smooth manifold and $f\colon M\to \R$ be a proper smooth
function bounded from below. Set $\CV_s:=H_*\big(\{f<s\};\F\big)$ with
the structure maps induced by the inclusions. No other requirements are
imposed on $f$, e.g., $f$ need not be Morse. However, it is not hard
to see that conditions \ref{PM1}--\ref{PM4} are met with $\CS$ being
the set of critical values of $f$. We note that one can have
$\dim \CV_s=\infty$ for $s\in\CS$ already when $M=S^1$, unless $f$
meets some additional conditions on $f$, e.g., that $f$ is real
analytic or the critical points of $f$ are isolated. Moreover, it is
easy to see that $\CV$ would be q-tame even when $f$ were just
continuous.
\end{Example}

Recall furthermore that an \emph{interval persistence module}
$\F_{(a,\,b]}$, where $-\infty<a<b\leq \infty$, is defined by setting
$$
{\F_{(a,\,b]}}_s:=\begin{cases}
  \F & \text{ when } s\in (a,\,b],\\
  0 & \text{ when } s \not\in (a,\,b],
\end{cases}
$$
and $\pi_{st}=\id$ if $a<s\leq t\leq b$ and $\pi_{st}=0$
otherwise. Interval modules are examples of simple persistence
modules, i.e., persistence modules that cannot be decomposed as a
(non-trivial) direct sum of other persistence modules.

A key fact which we will use in the paper is the normal form or
structure theorem asserting that every persistence module meeting the
above conditions can be decomposed as a direct sum of a countable
collection (i.e., a countable multiset) of interval persistence
modules. Moreover, this decomposition is unique up to re-ordering of
the sum. (In fact, conditions \ref{PM1}--\ref{PM4} are far from
optimal and can be considerably relaxed.) We refer the reader
\cite[Thm.\ 3.8]{BV} and \cite{CdSGO16} for proofs of this fact for
the class of persistence modules more general than considered in this
paper and further references, and to, e.g., \cite{CZCG, CB, ZC} for
previous or related results. (Here we only note that our class of
persistence modules directly fits in the framework of \cite[Thm.\
3.8]{BV}. Furthermore, while the interval decomposition fails for
q-tame modules in general, the barcode is still defined and the
isometry theorem holds in this case; see \cite{CdSGO16,Le}.)

This multiset $\CB(\CV)$ of intervals entering this decomposition is
referred to as the \emph{barcode} of $\CV$ and the intervals occurring
in $\CB(\CV)$ as \emph{bars}. The set of end-points of the bars is the
spectrum of the barcode. Then the normal form theorem gives a
one-to-one correspondence between isomorphism classes of persistence
modules and barcodes, i.e., multisets of intervals satisfying some
natural additional conditions corresponding to \ref{PM1}--\ref{PM4}.

One of these conditions is of particular interest to us.  For
$\eps>0$, denote by $\fb_\eps(\CV,s)$ or just $\fb_\eps(s)$ the number
of bars $(a,\,b]$ in $\CB(\CV)$ with $a<s$ of length $b-a> \eps$,
counted with multiplicity. This is the key numerical invariant of
persistence modules used in this paper. Interestingly, within this
framework the definition of $\fb_\eps(s)$ can be made formally
independent of the structure theorem: the multiplicity of an interval
as a bar in the barcode, and hence $\fb_\eps(s)$, can be defined by
purely linear algebra means without making use of the theorem; see
\cite[Sec.\ 2.1 and 3]{GGM}. If the multiplicity is zero, the interval
is not a bar in the barcode.

Clearly, $\fb_\eps(s)$ is locally constant in $s$ on the complement to
$\CS$ and in $\eps>0$ on the complement of
$\CS-\CS=\{s-s'\mid s,\, s'\in \CS\}$. The latter set is closed, but
can in general have positive measure. For instance, it can be an
interval; cf.\ \cite[Sec.\ 3.2.3]{CGG:Entropy} and \cite[p.\ 87]{GO}. 
\begin{Lemma}
  \label{lemma:b-eps}
  For every $\eps>0$ and $s\in\R$, we have $\fb_\eps(s)<\infty$.
\end{Lemma}

The lemma is quite standard (see, e.g., \cite[Cor.\ 3.5]{CdSGO16}) and
is in fact a consequence of just \ref{PM2}, and we include a proof
below only for the sake of completeness.

By the lemma, $\fb_\eps(s)$ is right semi-continuous in $\eps$ and
left semi-continuous in $s$, i.e.,
\begin{equation}
  \label{eq:bf-semicont}
  \fb_{\eps'}(s')=\fb_\eps(s)
\end{equation}
whenever $\eps'\geq \eps$ is close to $\eps$ and $s'\leq s$ is close
to $s$. Here we are using the condition that the inequalities in the
definition of $\fb_\eps(s)$ are strict.

\begin{proof}[Proof of Lemma \ref{lemma:b-eps}]
  Assume the contrary: for some $s$ and $\eps>0$, there exists an
  infinite sequence of intervals $I_i=(a_i,\, b_i]$ with $b_i<s$ and
  $b_i-a_i>\eps$. We will show that then there exists an interval
  $[t',\, t]$ contained in an infinite number of intervals
  $I_i$. Hence the rank of the map $\pi_{t't}$ is infinite which is
  impossible due \ref{PM2}.

  To this end, note that there are two, not necessarily mutually exclusive
  possibilities: an infinite subsequence of intervals is in
  $(-\infty, s]$ or an infinite subsequence of intervals contains $s$
  as an interior point. Passing to those subsequences we will treat
  these cases separately.

  In the former case, the intervals are actually contained in
  $(s_0,\, s]$ by \ref{PM4}. Passing to a subsequence again we can
  assume that $b_i\to b$ for some $b\in (s_0+\eps,\, s]$. Then for
  some $t'<t\leq b$, close to $b$, an infinite number of intervals
  contain $[t',\, t]$.  In the latter case, similarly there exists a
  pair of points $t'<t$ arbitrarily close to $s$ such that an infinite
  number of intervals contain $[t',\, t]$.
\end{proof}

It is worth keeping in mind that the number of all bars beginning
below $s$ can be infinite. We set $\fb_\eps(\infty)$ to be the total
number of bars of length greater than $\eps>0$. In general, this
number can also be infinite.

\subsubsection{Truncations and convergence}
Let us a call a persistence module $\CV$ \emph{bounded from above}
when $\CV_s=0$ for all sufficiently large $s$. Likewise, we say that a
barcode is bounded from above when all bars are finite and the
spectrum is bounded from above. The collections of bounded from above
persistence modules and barcodes carry natural finite distances, called
the \emph{interleaving distance} and, respectively, \emph{bottleneck
  distance} and the structure theorem gives an isometry; see
\cite{CdSGO16,Le,PRSZ} and references there. Furthermore, in this case
$\fb_\eps(\infty)<\infty$ for all $\eps>0$.  When the boundedness
requirement is relaxed and replaced by the condition that the spectrum
is bounded from above, the distances can be infinite but the structure
theorem still gives an isometry (the isometry theorem). The distance
between two persistence modules or barcodes is finite if and only if
they have the same number of infinite bars.

In the setting of Section \ref{sec:persistence-def}, the isometry
theorem still holds, but the interleaving distance between
persistence modules or equivalently the bottleneck distance
between their barcodes can be infinite even when the barcodes have the
same number of infinite bars.

The \emph{truncation} of $\CV$ at $s\in \R$ is the persistence module
$\Tr_s(\CV)$ defined by setting $\Tr_s(\CV)_t=\CV_t$ for $t\leq s$ and
$\Tr_s(\CV)_t=0$ for $t>s$ and modifying the structure maps in a
similar fashion. It is easy to see that
$$
\fb_\eps\big(\Tr_s(\CV),\infty\big)\leq
\fb_\eps(\CV,s)\leq\fb_\eps\big(\Tr_{s+\eps}(\CV),\infty\big)
\leq \fb_\eps(\CV,s+\eps).
$$
Here the second inequality is strict only when $\CV$ has bars of
length greater than $\eps$ beginning exactly at $s$. As a consequence,
when we are interested in the growth rate of $\fb_\eps(\CV,s)$ we can
replace it by $\fb_\eps\big(\Tr_s(\CV),\infty\big)$. Clearly,
truncation commutes with morphisms of persistence modules and there is
a natural map $\CV\to \Tr_s(\CV)$ (but not in the opposite direction).

Denote by $d$ the interleaving distance, which we allow to take an
infinite value. For every $s\in\R$, the function
$$
d_s(\CV,\CW):=d\big(\Tr_s(\CV), \Tr_s(\CV)\big)
$$
is a pseudo-distance. Clearly, $d_s\leq d_t$ when $s\leq t$. This
increasing family of functions gives rise to a metrizable topology on
the space of isomorphism classes of persistence modules. (It is
sufficient to consider the family $d_{s_k}$ for any set $\{s_k\}$
unbounded from above, e.g., a sequence $s_k\to\infty$, to get the same
topology.) In other words, for a persistence module $\CW$ let $\CU_{\eta,s}$ be
the set of all $\CV$ such that the distance between $\Tr_s(\CV)$ and
$\Tr_s(\CW)$ is less than $\eta$.  Then the collection $\CU_{\eta,s}$,
for all $s\in \R$ and $\eta>0$, is a local base at $\CW$ for this
topology.

Convergence in this topology is described as follows.  Consider a
family of persistence modules $\CV(i)$ parametrized by a partially
ordered set $I$ containing a cofinal genuinely ordered subset $A$. We
say that the family $\CV(i)$ \emph{converges} to $\CW$ if for every
$s$, the family of truncations $\Tr_s(\CV(i))$ converges to
$\Tr_s(\CW)$ with respect to the interleaving distance, i.e., for
every $s$ and every $\eta>0$ there exists $i_s(\eta)\in I$ such that
the distance between $\Tr_s(\CV(i))$ and $\Tr_s(\CW)$ is less than
$\eta$ whenever $i\in A$ and $i>i_s(\eta)$.  Note that this
convergence is not necessarily uniform, i.e., $i_s(\eta)$ can depend
on $s$, and that the notion of convergence is independent of the
choice of cofinal ordered subsystem $A$. (However, such a system must
exist for convergence to make sense.) The limit $\CW$ of a system
$\CV(i)$ is necessarily unique up to isomorphism.

\begin{Remark}
  \label{rmk:conv-spaces}
  In general, convergence of persistence modules $\CV(i)$ to $\CW$ does
  not imply that the underlying vector spaces $\CV(i)^s$ converge for
  every $s$ in any sense. For instance, fix a persistence module $\CW$
  and a monotone decreasing sequence $\eps_i\to 0+$. Then the
  sequences of shifted persistence modules
  $\CV^{+\eps_i}_s:=\CW_{s+\eps_i}$ and
  $\CV^{-\eps_i}_s:=\CW_{s-\eps_i}$ converge to $\CV$. However, while
  there are natural maps $\CV^{-\eps_i}\to \CW$ and the space $\CW_s$
  is the direct limit of the spaces $\CV^{-\eps_i}_s$, for
  $\CV^{+\eps_i}$ the maps go in the opposite direction and, moreover,
  the space $\CW_s$ need not be the inverse limit of
  $\CV^{+\eps_i}_s$. This happens already when $\CW=\F_{(a,\,b]}$ and
  $s=a$. Moreover, without a finiteness-type assumption this can also
  happen when $s\not\in \CS(\CW)$. Let, for instance,
  $\CV(i)=\F_{(a_i,\,b_i]}$ where the intervals $(a_i,\,b_i]$ are
  nested and shrink to the only common point $s$ of their interiors
  $(a_i,\,b_i)$. Then the limit of $\CV(i)$ is the zero persistence
  module. In particular, $\CW_s=0$ while $\CV(i)_s=\F$ and $s$ is
  outside the spectra of all the modules involved.
\end{Remark}  

Fix $s$ and $\eps>0$, and also $\delta>0$. Then, for a persistence
module $\CW$, we have
\begin{equation}
  \label{eq:dist-fb2}
\fb_{\eps}(\CV,s)\geq\fb_\eps(\CW,s)
\end{equation}
for any persistence module $\CV$ when $\Tr_{s+\delta}(\CV)$
is sufficiently close to $\Tr_{s+\delta}(\CW)$, depending on $s$,
$\eps$, $\delta$ and $\CW$. In other other words, $\fb_\eps(\CW,s)$ is
lower semicontinuous in $\CW$. This is a consequence of the condition
that the inequalities in the definition of $\fb_\eps(s)$ are strict
and the fact that $\fb_\eps(s)$ is finite for all $s$.

\subsection{Symplectic homology}
\label{sec:SH}
The notion of symplectic homology goes back to \cite{CFH,Vi}. In this
section we briefly review the definition and properties of filtered
symplectic homology from a persistent homology perspective, following
closely \cite{CGGM:Reeb}. Most of the material here is quite
standard. The only new point is the definition of the barcode growth
function $\fb_\eps(s)$ for star-shaped domains with not-necessarily
smooth boundary; see Section \ref{sec:C0}.  We also touch upon the
symplectic Banach--Mazur distance in Section \ref{sec:SBM}, originally
introduced by Ostrover and L. Polterovich (without the unknottedness
requirement) and then refined and studied in \cite{PRSZ, SZ-SBM,
  Us-SBM}. In Section \ref{sec:Inv} we discuss invariance of the
filtered symplectic homology. Some of these results are also new.

\subsubsection{Conventions and notations}
\label{sec:setting}
Let us begin by spelling out our conventions and notation on the
symplectic dynamics side, which are essentially identical to the ones
used in \cite{CGG:Mult, CGGM, CGGM:Reeb, GG:LS}. 

Let, as in Section \ref{sec:results}, $\alpha$ be the contact form on
the boundary $M=\p W$ of a Liouville domain $W^{2n\geq 4}$. We will
also use the same notation $\alpha$ for a primitive of the symplectic
form $\omega$ on $W$. The grading of Floer or symplectic homology is
inessential for our purposes and we make no assumptions on $c_1(TW)$.
As usual, denote by $\WW$ the symplectic completion of $W$, i.e.,
$$
\WW=W\cup_M M\times [1,\,\infty)
$$
with the symplectic form $\omega=d\alpha$ extended from $W$ to
$M\times [1,\infty)$ as
$$
\omega := d(r\alpha),
$$
where $r$ is the coordinate on $[1,\,\infty)$. Sometimes it is
convenient to have the function $r$ also defined on a collar of
$M=\p W$ in $W$. Thus we can think of $\WW$ as the union of $W$ and
$M\times [1-\eta,\,\infty)$ for small $\eta>0$ with
$M\times [1-\eta,\, 1]$ lying in $W$ and the symplectic form given by
the same formula.

Unless specifically stated otherwise, most of the Hamiltonians
$H\colon \WW\to \R$ considered in this paper are constant on $W$ and
depend only on $r$ outside $W$, i.e., $H=h(r)$ on
$M\times [1,\,\infty)$, where the $C^\infty$-smooth function
$h\colon [1,\,\infty)\to \R$ is required to meet the following three
conditions:
\begin{itemize}
\item $h$ is strictly monotone increasing;
\item $h$ is convex, i.e., $h''\geq 0$, and $h''>0$ on $(1,\, \rmax)$
  for some $\rmax>1$ depending on~$h$;
\item $h(r)$ is linear, i.e., $h(r)=ar-c$, when $r\geq \rmax$.
\end{itemize}
In other words, the function $h$ changes from a constant on $W$ to
convex in $r$ on $M\times [1,\, \rmax]$, and strictly convex on the
interior, to linear in $r$ on $M\times [\rmax,\, \infty)$.

We will refer to $a$ as the \emph{slope} of $H$ (or $h$) and write
$a=\slope(H)$. The slope is often, but not always, assumed to be
outside the action spectrum of $\alpha$, i.e.,
$a\not\in\CS(\alpha)$. We call $H$ \emph{admissible} if
$H|_W=\const<0$ and \emph{semi-admissible} when $H|_W\equiv 0$. (This
terminology differs from the standard usage, and we emphasize that
\emph{admissible Hamiltonians are not semi-admissible}.) When $H$
satisfies only the last of the three conditions, we call it
\emph{linear at infinity}.

The difference between admissible and semi-admissible Hamiltonians is
just an additive constant: $H- H|_W$ is semi-admissible when $H$ is
admissible. Hence the two Hamiltonians have the same filtered Floer
homology up to an action shift. For our purposes, semi-admissible
Hamiltonians are notably more suitable due to the $H|_W\equiv 0$
normalization.

The Hamiltonian vector field $X_H$ is determined by the condition
$$
\omega(X_H,\, \cdot)=-dH,
$$
and, on $M\times [1,\,\infty)$, 
$$
X_H=h'(r) R_\alpha,
$$
where $R_\alpha$ is the Reeb vector field. We denote the Hamiltonian
flow of $H$ by $\varphi_H^t$, the Reeb flow of $\alpha$ by
$\varphi_\alpha^t$, where $t\in \R$, and the Hamiltonian
diffeomorphism generated by $H$ by $\varphi_H:=\varphi_H^1$.

Every $T$-periodic orbit $z$ of the Reeb flow with $T<a=\slope(H)$
gives rise to a 1-periodic orbit $\tz=(z,r_*)$ of the flow of $H$ with
$r_*$ determined by the condition
\begin{equation}
  \label{eq:level}
h'(r_*)=T.
\end{equation}
Clearly, $\tz$ lies in the shell $1<r<\rmax$, and we have a one-to-one
correspondence between 1-periodic orbits of $H$ and the periodic
orbits of $\varphi_\alpha^t$ with period $T<a$ whenever
$a\not\in \CS(\alpha)$.

The action functional $\CA_H$ is defined by
$$
\CA_H(\gamma)=\int_\gamma\hat{\alpha}-\int_{S^1} H(\gamma(t))\, dt,
$$
where $\gamma\colon S^1=\R/\Z\to \WW$ is a smooth loop in $\WW$ and
$\hat{\alpha}$ is the Liouville primitive $\alpha$ of $\omega$ on $W$
and $\hat{\alpha}=r\alpha$ on $M\times [1-\eta,\,\infty)$ for a
sufficiently small $\eta>0$. More explicitly, when
$\gamma\colon S^1\to M\times [1,\,\infty)$, we have
$$
\CA_H(\gamma)= \int_{S^1} r(\gamma(t))\alpha\big(\gamma'(t)\big)\, dt
- \int_{S^1} h\big(r(\gamma(t))\big)\, dt.
$$
When $H'\leq H$ are two Hamiltonians linear at infinity we have
naturally defined continuation maps $\HF^s(H')\to \HF^s(H)$.

\subsubsection{Basic definitions and facts}
\label{sec:SH-def}
Denote by $\HF^s(H)$ the filtered Floer homology of a Hamiltonian
$H$ on $\WW$ which we only require to be linear at infinity for this
homology to be defined and have usual properties.  The
\emph{symplectic homology} $\SH^s(\alpha)$, where $s>0$, is
defined as
\begin{equation}
  \label{eq:SH}
\SH^s(W):=\varinjlim_H \HF^s(H),
\end{equation}
where traditionally the limit is taken over all Hamiltonians linear at
infinity and such that $H|_W<0$. 

In \eqref{eq:SH} we could have required that $H|_W\leq 0$ rather than
that $H|_W<0$, or equivalently allowed $H$ to be semi-admissible or
admissible. This would result in the same groups $\SH^s(W)$.  Indeed,
since admissible (but not semi-admissible) Hamiltonians form a
co-final family, we can limit $H$ to this class. Furthermore, let $H$
be semi-admissible.  Pick two sequences of positive numbers:
$a_i\to\infty$ and $a_i\to 0$. Then the sequence $H_i=a_i H-\eps_i$ is
co-final in the class of admissible Hamiltonians.

The filtered symplectic homology spaces $\SH^s(W)$ form a persistence
module, denoted here by $\SH(W)$ or $\SH(\alpha)$, in the sense of
Section \ref{sec:persistence} with $\CS=\{0\}\cup\CS(\alpha)$; see
\cite{CGGM:Reeb}. (Note that here we do not use the logarithmic scale
$\ln s$ as in, e.g., \cite{PRSZ} or \cite{Hu}.)  Hence we have the
barcode growth function $\fb_\eps(s)$ associated with a Liouville
domain $W$. When we need to emphasize the dependence of $\fb_\eps(s)$
on $(W,\alpha)$ we will write $\fb_\eps(W,s)$.

The definition of symplectic homology via a direct limit,
\eqref{eq:SH}, over admissible or even semi-admissible Hamiltonians is
quite inconvenient for our purposes. It turns out that the vector
space $\SH^s(W)$ can be directly identified with the Floer homology
of a suitable semi-admissible Hamiltonian.

\begin{Proposition}[Cor.\ 3.7, \cite{CGGM:Reeb}]
 \label{prop:sympl-Fl} 
 For any semi-admissible Hamiltonian $H$ with 
 $\slope(H)=s$,
 $$
\SH^s(\alpha)\cong \HF(H).
$$
Moreover, whenever $H'\leq H$ are semi-admissible with $s'=\slope(H')$
and $s=\slope(H)$, the diagram
$$
\begin{tikzcd}[row sep=large]
&\SH^{s'}(\alpha)
\arrow[r,"\cong"]
\arrow[d]
&
\HF(H')
\arrow[d]
\\
&
\SH^{s}(\alpha)
\arrow[r,"\cong"]
&
\HF(H)
\end{tikzcd}
$$
commutes, where the left vertical arrow is the structure or
``inclusion'' map and the right vertical arrow is the continuation
map.
\end{Proposition}

Fixing $s>0$, consider now a $C^\infty$-small non-degenerate perturbation $\talpha$
of $\alpha$ and a 1-periodic in time perturbation $\tH$ of a
semi-admissible Hamiltonian $H$ for $\talpha$ such that $\tH$ is still
linear at infinity with slope $s=\slope(H)$. Then, with a generic
almost complex structure $J$ fixed, we have the Floer complex of $\tH$
defined, which we will refer to as a \emph{symplectic homology
  complex} with action less than or equal to $s$. By Proposition
\ref{prop:sympl-Fl}, the homology of this complex is $\SH^s(W)$.
We will call the perturbations and $J$ \emph{auxiliary structures}.

\begin{Remark}
  We do not claim that $\SH^t(\alpha)=\HF^t(H)$ for all
  $t\leq s$ in the setting of Proposition
  \ref{prop:sympl-Fl}. However, we can ensure that
  $\SH^t(\alpha)=\HF^{\fa(t)}(H)$ for a monotone increasing function
  $\fa$ which can be made $C^0$-close to the identity by a suitable
  choice of $H$, approximating the continuous Hamiltonian which is
  identically zero on $W$ and the linear function $s(r-1)$ on
  $M\times [1,\infty)$; see \cite[Thm.\ 3.5]{CGGM:Reeb}. Furthermore,
  it is worth keeping in mind that the reparametrized persistence
  module $\HF^{\fa(t)}(H)$ is not the truncation
  $\Tr_s(\SH(\alpha))$. The reason is that short bars passing through
  $s$ get cut at $s$ in the truncation but give rise to infinite bars
  in the Floer homology.
\end{Remark}  

\subsubsection{Symplectic Banach--Mazur distance}
\label{sec:SBM}
The symplectic Banach--Mazur distance is a symplectic analogue of the
Banach--Mazur distance in convex geometry. It can also be thought of
as an analogue of the Hofer distance for Reeb flows. Several variants
of the distance have been introduced recently and the notion is quite
nuanced; see \cite{PRSZ, SZ-SBM,Us-SBM} and references therein. Here
we closely follow \cite{PRSZ}. The most essential departure from that
treatment is that here as in \cite{SZ-SBM,Us-SBM} we do not require
Liouville maps to be defined globally and be compactly
supported. These conditions are not essential with our definition of
the symplectic homology. In that sense our approach is more similar to
that in \cite{Us-SBM} where the distance is defined for open domains.

For the sake of simplicity, we limit our attention to bounded
star-shaped domains in $\R^{2n}$. However, the construction and basic
results extend word-for-word to other Liouville manifolds of finite
type with the Liouville flow for time $\ln \lambda$ utilized to
define the rescaling $\lambda U$ of $U$; see \cite{PRSZ}.

We start with several auxiliary definitions and notation. For a
continuous function $f\colon S^{2n-1}\to \R$, we set $W_f$ to be the
star-shaped domain $\{r\leq f(\theta)\}\subset \R^{2n}$. We will adopt
this as the definition of a closed bounded star-shaped domain.  This
is consistent with the notation from Section \ref{sec:toric} with $f$
used in place of the composition $f\circ \mu$.

Let $K$ be a compact set and $U$ an open subset of $\R^{2n}$. An exact
symplectic embedding from $K$ to $U$ is a smooth exact
symplectomorphism from a small neighborhood of $K$ into $U$. In our
setting all domains are contractible, and hence the exactness requirement
is satisfied automatically. For $\lambda\in (0,\infty)$, we will also
use the notation $\lambda$ for the rescaling map $z\mapsto \lambda z$
(the Liouville flow) and for a subset $U\subset \R^{2n}$ we set
$\lambda(U)=\{\lambda z\mid z\in U\}$.

Let $U$ and $W$ be two closed compact star-shaped domains with not
necessarily smooth boundary. Denote by $\mathring{U}$ and
$\mathring{W}$ their interiors. We say that $\lambda>1$ is
\emph{$(U,W)$-admissible} if there exist (automatically exact)
symplectic embeddings $\varphi\colon \lambda^{-1} U\to \mathring{W}$
and $\psi\colon W\to \lambda \mathring{U}$ such that the composition
$\psi\varphi\colon \lambda^{-1} U\to \lambda\mathring{U}$ is isotopic
to the natural inclusion in the class of (exact) symplectic embeddings
of $\lambda^{-1} U$ into $\lambda\mathring{U}$. The latter isotopy
condition is referred to as \emph{unknottedness}, cf.\ \cite{GU,Us-SBM}.) We set
$$
\delta(U,W):=\inf\big\{\ln \lambda\mid \lambda \text{ is 
    $(W,U)$-admissible}\big\}.
  $$
  The unknottedness requirement is essential but makes $\delta(U,W)$
  asymmetric; cf.\ \cite{GU,Us-SBM}. By definition, the
  \emph{symplectic Banach--Mazur} distance between $U$ and $W$ is
$$
\dSBM(U,W):=\max\{\delta(U,W), \delta(W,U)\}.
  $$

  For instance, $\dSBM(U,\lambda U)=|\ln \lambda|$. One can show that
  $\dSBM$ is indeed a pseudo-distance on the set of closed compact
  star-shaped domains and, in particular, the triangle inequality
  holds. However, $\dSBM$ is only a pseudo-distance: it is easy to see
  that $\dSBM(W,W')=0$ whenever $W$ and $W'$ are (exact)
  symplectomorphic. (When $W$ and $W'$ are not smooth, this
  requirement should be read as that their small open neighborhoods
  are exact symplectomorphic.) We will refine this observation in
  Lemma \ref{lemma:dist}.  Furthermore, $\dSBM$ is an invariant of
  compactly supported symplectomorphisms of $\R^{2n}$.

  The feature of the symplectic Banach--Mazur distance essential for
  us is that on logarithmic scale it bounds the interleaving distance
  between symplectic homology persistence modules for star-shaped
  domains with smooth boundary, and hence the
  bottleneck distance between their barcodes; see \cite[Thm.\
  9.4.7]{PRSZ} and also \cite{GU,SZ-SBM,Us-SBM} for closely related
  results. Thus, with our conventions
  \begin{equation}
    \label{eq:int-SBM}
    d_s\big(\SH(W),\SH(U)\big)
    :=d\big(\Tr_s(\SH(W)),\Tr_s(\SH(U))\big)\leq  s\big(e^{\dSBM(W,U)}-1\big),
  \end{equation}
  when $W$ and $U$ are smooth. As a consequence of \eqref{eq:dist-fb2},
    \begin{equation}
    \label{eq:dist-fb3}
    \fb_\eps(U,s)\geq \fb_\eps(W,s)
\end{equation}
when $U$ is sufficiently $\dSBM$-close to $W$, depending on $W$,
$\eps$ and $s$.  In particular, for smooth
domains, 
   \begin{equation}
    \label{eq:dist-fb4}
    \fb_\eps(U,s) = \fb_\eps(W,s) \textrm{ when }\dSBM(U,W)=0.
  \end{equation}

Furthermore, fix $\eps>0$, $0<\delta<\eps$ and $s$. Then
\begin{equation}
  \label{eq:fb-mon1}
\fb_{\eps+\delta}(W,s-\delta)\leq \fb_\eps(U,s)\leq \fb_{\eps-\delta}(W,s+\delta)
\end{equation}
for any two star-shaped domains $U$ and $W$ with smooth boundary such that
$\dSBM(U,W)$ is sufficiently small and by symmetry
\begin{equation}
  \label{eq:fb-mon2}
\fb_{\eps+\delta}(U,s-\delta)\leq \fb_\eps(W,s)\leq \fb_{\eps-\delta}(U,s+\delta).
\end{equation}

As a particular case of \eqref{eq:int-SBM} and \eqref{eq:dist-fb3}, we have

  $$
  d_s\big(\SH(W_f),\SH(W_g)\big)
  :=d\big(\Tr_s(\SH(W_f)),\Tr_s(\SH(W_g))\big)\leq s(e^\eta-1),
  $$
  where $f$ and $g$ are smooth and $\max |\ln (f/g)|<\eta$, and
$$
  \fb_\eps(W_g,s)\geq \fb_\eps(W_f,s)
$$
when $\max |\ln (f/g)|$ is small, depending on $f$, $\eps$ and $s$.

\subsubsection{Non-smooth domains}
\label{sec:C0}  
Assume now that $W=W_f$, where $f$ is just continuous. (As in the
previous section, we do not require $W$ to be toric.) Set $U=W_{g_k}$,
where the functions $g_k$ are smooth and $g_k\stackrel{C^0}{\to}
f$. Then $\dSBM (W_{g_k},W_f)\to 0$. More specifically,
$\dSBM (W_{g_k},W_f) \leq \max|\ln(f/g_k)|$. In particular, any
star-shaped domain can be approximated by star-shaped domains with
smooth boundary and the sequence $W_{g_k}$ is Cauchy with respect to
$\dSBM$. Moreover, when needed, we can require that $g_k<f$, and hence
$W_{g_k}\subset \mathring{W}_f$.

Then the filtered symplectic
homology $\SH^s(W_f)$ can be defined by continuity, but we see no
reason why these spaces should form a persistence module in the sense
of Section \ref{sec:persistence-def}. Instead, we do this for
$\fb_\eps$. Namely, set
\begin{equation}
  \label{eq:fb-C0}
\fb_\eps(W,s):=\liminf_{U} \fb_\eps(U,s),
\end{equation}
where the lower limit as taken over all star-shaped domains with
smooth boundary as $\dSBM(U,W)\to 0$. By \eqref{eq:fb-mon1},
$\fb_\eps(W,s)<\infty$. More explicitly, \eqref{eq:fb-C0} should be
read as
$$
\fb_\eps(W,s)=\inf_{\{U_k\}}\lim_{k\to \infty} \fb_\eps(U_k,s),
$$
where $\{U_k\}$ ranges through all sequences of star-shaped domains with
smooth boundary such that $\dSBM(U_k, W)\to 0$ and
$\fb_\eps(U_k,s)$ converges. The infimum is attained and can be replaced by the
minimum. Equivalently,
$$
\fb_\eps(W,s)=\sup_{\delta>0}\inf\big\{\fb_\eps(U,s)\mid \textrm{$U$ is
  smooth and $\dSBM(W,U)<\delta$}\big\}.
$$
When $W$ is also smooth, this definition agrees with the original one
by \eqref{eq:dist-fb3} or \eqref{eq:dist-fb4}.  Note also that
$\fb_\eps(W,s)<\infty$ by, e.g., \eqref{eq:fb-mon2} even when $W$ is
not smooth.

As an immediate consequence of the definition,
\begin{equation}
  \label{eq:dSBM-fb}
  \fb_\eps(W,s)=\fb_\eps(W',s)\textrm{ when } \dSBM(W,W')=0.
\end{equation}
Clearly,
\begin{equation}
  \label{eq:fb-C01}
\fb_\eps(W_f,s)\leq \liminf_{g\to
  f}\fb_\eps(W_g,s),
\end{equation}
where $g$ is $C^\infty$-smooth and $g\stackrel{C^0}{\to} f$. We do not
know if this inequality can be strict, i.e., if it is not always
equality. The reason for choosing the lower limit in \eqref{eq:fb-C0}
is the proof of Theorem \ref{thm:toric-nonsmooth} giving an upper
bound on $\fb_\eps$ for a particular collection of smooth
approximations. With \eqref{eq:fb-C01}, we have the barcode entropy
$\hbar(W)$ of $W$ and the low barcode entropy
$\hbar^{\scriptstyle{Slow}}(W)$ defined by \eqref{eq:hbar} and
\eqref{eq:hbar-slow}. However, there is no reason to think that
$\hbar(W)$ is always finite when $W$ is not smooth.

For domains with non-smooth boundary, $\fb_\eps(W,s)$ inherits
some monotonicity and continuity properties but these are not as
strong as for smooth domains. Namely, by using \eqref{eq:fb-mon1} and
\eqref{eq:fb-mon2}, it is not hard to show that
$$
\sup_{\eps'>\eps}\fb_{\eps'}(W, s_0)\leq \fb_\eps(W, s_1)\textrm{ when
} s_0<s_1
$$
and 
$$
\sup_{s'<s}\fb_{\eps_0}(W, s')\leq \fb_{\eps_1}(W, s)\textrm{ when
} \eps_1<\eps_0.
$$ 
We do not know if $\fb_\eps(W,s)$ is monotone in $\eps$ for $s$ fixed
and in $s$ for $\eps$ fixed when $W$ is not smooth.

Likewise, we do not know if the limit exists in \eqref{eq:fb-C0} or
how far the upper limit can be from the lower limit when $W$ is not
smooth. However, from again \eqref{eq:fb-mon1} and
\eqref{eq:fb-mon2} it is easy to see that
\begin{equation}
  \label{eq:fb-infi-sup}
\limsup_{U\to W} b_{\eps'}(W, s')\leq \fb_\eps(W,s)
\textrm{ when
} \eps'>\eps\textrm{ and } s'<s.
\end{equation}
As a consequence, replacing the lower limit by the upper
limit in \eqref{eq:fb-C0} would not affect the barcode
entropy and the slow barcode entropy of $W$, but might affect
$\eps$-entropies.

\begin{Remark}
  \label{rmk:limit-barcode}
  There is some flexibility in defining $\fb_\eps(s)$ as a limit and
  the definition above is perhaps the most naive. A viable alternative
  to \eqref{eq:fb-C0} would be to first define the barcode of $W$ as
  in \cite{BHS,LRSV,KS} as the limit over smooth approximations by
  using the symplectic Banach--Mazur distance in place of the Hofer or
  spectral norm. The resulting barcode would then belong to a certain
  completion of the space of ``locally finite'' barcodes with respect
  to the sequence of truncated bottleneck distances. It would have a
  well-defined growth function $\fb_\eps(s)$ which could be different
  from \eqref{eq:fb-C0} and could possibly have better continuity or
  invariance properties. However, this would be a more involved
  construction and we believe that, by \eqref{eq:fb-infi-sup}, the
  function would grow at the same pace as \eqref{eq:fb-C0}, adjusting
  for $\eps$, and give rise to exactly the same (slow) entropy. It
  would be interesting though to see what the ``action spectrum'' of
  such a limiting barcode could look like when the function $f$
  defining $W$ is pathologically non-smooth.
\end{Remark}  

\subsection{Invariance}
\label{sec:Inv}
The filtered symplectic homology is an invariant of the exact
symplectomorphism class of the interior of the Liouville domain. The
first variant of this statement was proved in \cite{CFHW}, and since
then the fact has been revisited several times in different contexts;
see, e.g., \cite{Gu,Hu}. Here we follow the suit.

\begin{Proposition}
  \label{prop:Inv1}
  Let $(W, \alpha)$ and $(W',\alpha')$ be Liouville domains with
  smooth boundary. Assume that the interiors $\mathring{W}$ and
  $\mathring{W}'$ are exact symplectomorphic.  Then
$$
\SH(W)=\SH(W')
$$
as persistence modules.
\end{Proposition}

One consequence of the proposition and the results in \cite{CGGM:Reeb}
is, for instance, that in dimension four the topological entropy of the
Reeb flow on $\p W$ is an invariant of the interior up to exact
symplectomorphisms; see Corollary \ref{cor:htop}.

Note that the exactness requirement that can be dropped if we restrict
the symplectic homology to periodic orbits contractible in the
domain. The proof of the proposition, which we outline below, is
somewhat different from the proofs of other similar statements in
that, for instance, it does not use Viterbo transfer even
implicitly. However, since this material is rather standard and well
known to experts we only briefly sketch the argument.

\begin{proof}[Outline of the proof]
  Let $W$ be a Liouville domain with smooth boundary $M=\p W$.  Fix a constant
  $C>0$ and consider the class $\CH_C$ of autonomous Hamiltonians $H$
  on $W$ meeting the following requirements:
  \begin{itemize}
\item $H\equiv C$ near $M$ and  (in particular, $H$ is constant near
  $M$); and

\item $0\leq H\leq C=\max H$.  
  \end{itemize}    
  As a consequence of these requirements -- in fact, the first one
  would be sufficient -- the filtered Floer homology $\HF^s(H)$ is
  defined, at least when $s< C$, and independent of the background
  choice of the almost complex structure; see \cite{CFH}. Moreover,
  for $H_1\geq H_0$ in $\CH_C$ we have a well-defined continuation
  homomorphism $\HF^s(H_0)\to \HF^s(H_1)$ with standard functoriality
  properties. Arguing as in \cite[Sec.\ 3]{CGGM}, it is not hard to
  show that
  \begin{equation}
    \label{eq:SH2}
    \SH^s(W)=\varprojlim_{H\in \CH_C} \HF^s(H),
  \end{equation}
  provided that $s<C$. The right had side of \eqref{eq:SH2}
  is invariant under exact symplectomorphisms of the interior of
  $W$. (The exactness guarantees that the action filtration is
  preserved. An additional point of \eqref{eq:SH2} not used here is
  that the right hand is independent of $C$.) The proposition follows.
  \end{proof}

  \begin{Remark}
    \label{rmk:SH-C0}
    The right hand side of \eqref{eq:SH2} is the definition of the
    filtered symplectic homology used in \cite{CFHW}. Accounting for
    the differences in conventions, it is also essentially equivalent
    to the one adopted in \cite{PRSZ}. One could also employ it to
    define the filtered symplectic homology of $W$ when $W$ is not
    smooth. Equivalently, the homology in the non-smooth case could be
    defined as the inverse limit of the symplectic homology over exact
    embeddings of smooth domains into $\mathring{W}$ as in
    \cite{Us-SBM}. The resulting homology would be an invariant of
    exact symplectomorphisms of the interior. However, as we have
    pointed out, it would not, most likely, be a persistent module in
    the sense of Section \ref{sec:persistence-def}, but rather an
    element of a certain completion of the space of persistence
    modules; cf. \cite{AI,Hu,LRSV,Us-SBM}. It is not clear what its
    spectrum could be and how to define its barcode. (This approach
    might or might not be equivalent to the one described in Remark
    \ref{rmk:limit-barcode}.)  Here we have preferred to get around
    this difficulty by using \eqref{eq:fb-C0}.
\end{Remark}    

By Proposition \ref{prop:Inv1}, $\fb_\eps(W,s)=\fb_\eps(W',s)$ when
$W$ and $W'$ are smooth and their interiors are exact
symplectomorphic.  We do not know if, as stated, this is also true for
non-smooth star-shaped domains, but we have the following invariance
result which shows that the domains have the same $\fb_\eps(s)$ when
one of the domains is smooth and that in general these functions have
the same growth rate with an adjustment for $\eps$. This is sufficient
for our purposes.

\begin{Theorem}
  \label{thm:inv-nonsmooth}
Let $W$ and $W'$ be closed, not necessarily smooth, star-shaped
domains. Assume that their interiors $\mathring{W}$ and
$\mathring{W'}$ are (exact) symplectomorphic. Then
\begin{itemize}

\item $\fb_\eps(W,s)=\fb_\eps(W',s)$ for all $\eps>0$ and all $s$ when
  one of the domains is smooth, and
  \item for any $0<\delta<\eps$ and all $s>\delta$ we have
    $$
    \fb_{\eps+\delta}(W,s-\delta)\leq \fb_{\eps}(W',s)\textrm{ and }
    \fb_{\eps+\delta}(W',s-\delta)\leq \fb_{\eps}(W,s).
    $$
  \end{itemize}
\end{Theorem}

\begin{Corollary}
  \label{cor:inv-nonsmooth}
  In the setting of Theorem \ref{thm:inv-nonsmooth}, $W$ and $W'$ have
  the same barcode entropy and slow barcode entropy.
\end{Corollary}

The key to the proof is the following observation.
  \begin{Lemma}
    \label{lemma:dist}
    As in Theorem \ref{thm:inv-nonsmooth}, let $W$ and $W'$ be closed,
    not necessarily smooth, star-shaped domains such that the
    interiors $\mathring{W}$ and $\mathring{W'}$ are (exact)
    symplectomorphic. Then 
    $\dSBM(W, W')= 0$.
  \end{Lemma}

  We recall in this connection that $\dSBM$ is only a pseudo-metric,
  and hence a sequence $\dSBM$-converging can have more than one
  limit. The proof of the lemma, given later in this section, is
  simple but not entirely automatic due to the fact that the
  definition of the distance involves the entire domains but not only
  their interiors.

  Applying the triangle inequality and using symmetry in $W$ and $W'$
  we obtain the following.

  \begin{Corollary}
    \label{cor:dist}
  Let $W$ and $W'$ be as in the lemma and let $U$ be another
  star-shaped domain. Then $\dSBM(U,W) =\dSBM(U,W')$.
\end{Corollary}

\begin{proof}[Proof of Theorem \ref{thm:inv-nonsmooth}]
  The first assertion is a consequence of \eqref{eq:dSBM-fb} and Lemma
  \ref{lemma:dist}.  To prove the second one, pick two sequences $U_k$
  and $U'_k$ of domains with smooth boundary such that
  $\dSBM(U_k, W)\to 0$ and $\fb_\eps(W,s)=\lim \fb_\eps(U_k,s)$ as
  $k\to\infty$, and similarly $\dSBM(U'_k, W')\to 0$ and
  $\fb_\eps(W',s)=\lim \fb_\eps(U'_k,s)$. Then $\dSBM(U_k,U'_k)\to 0$
  by the triangle inequality. Now the assertion readily follows from
  \eqref{eq:fb-mon1} and \eqref{eq:fb-mon2}. Indeed, for instance,
$$
\fb_{\eps+\delta}(W,s-\delta)\leq \liminf_{k\to \infty} \fb_{\eps+\delta}(U_k,s-\delta)
\leq \lim_{k\to\infty} \fb_{\eps}(U'_k,s)=\fb_\eps(W',s).
$$
\end{proof}

To complete the proof of the theorem it remains to prove the lemma.

\begin{proof}[Proof of Lemma \ref{lemma:dist}] It suffices to show
  that $\delta(W,W')=0$. Then $\delta(W',W)=0$, for $W$ and $W'$ play
  symmetric roles.

  Denote by $F$ an exact symplectomorphism $\mathring{W}\to
  \mathring{W}'$. Pick any $\lambda>1$. The exact symplectic embedding
  $\varphi\colon \lambda^{-1}W\to \mathring{W}'$ is simply the
  restriction of $F$ to $\lambda^{-1}W$. To define $\psi\colon W'\to
  \lambda W$, pick $1<\xi<\lambda$ so close to 1 that $F(\lambda^{-1}W)\subset
  \xi^{-1}\mathring{W}'$. Then $\psi$ is the composition
  $$
  \psi\colon
  W'
  \stackrel{\xi^{-1}}{\longrightarrow}
  \mathring{W}'
  \stackrel{F^{-1}}{\longrightarrow}
  \mathring{W}
  \stackrel{\xi}{\longrightarrow}
  \xi\mathring{W}
  \stackrel{\iota}{\longrightarrow}
  \lambda\mathring{W},
  $$
  where the last map $\iota$ is the inclusion. Then the composition
  $\psi\varphi$ is defined and
  $$
  \psi\varphi=\iota\circ (\xi F^{-1}\xi^{-1} F)\colon \lambda^{-1}W\to
  \lambda \mathring{W}.
  $$
  In contrast with $\psi$, this map is defined for all $\xi$ in the
  range between 1 and the original value of $\xi$, with 1
  included. Thus varying $\xi$ in this range we obtain the desired
  homotopy between $\iota$ and $\psi\varphi$. In other words, as a
  function of $t\in [1,\xi]$ and $x\in\lambda^{-1} W$, the homotopy is
  $t\big(F^{-1}(tF(x))\big)$, which is clearly a smooth function of
  $(t,x)$ and an exact symplectic embedding. For $t=\xi$ this is
  $\psi\varphi$ and for $t=1$ this is the inclusion $\iota$.
\end{proof}

\section{Proofs of Theorems \ref{thm:toric-Reeb} and \ref {thm:toric-nonsmooth}}
\label{sec:main-prfs}

\subsection{Proof of Theorem \ref{thm:toric-Reeb}}
\label{sec:pf-toric-Reeb}
Our next goal is to prove Theorem \ref{thm:toric-Reeb}. As in Section
\ref{sec:pf-toric-sympl}, we will prove a stronger result, giving an
upper bound on the number of generators of the symplectic homology
complex and hence on $\fb_\eps(s)$.  Let, as in Section
\ref{sec:results-Reeb}, $\Omega$ be a star-shaped domain with a smooth
boundary, which we require to be strictly star-shaped, i.e., the
radial direction is nowhere tangent to $\p\Omega$ including the
boundary of the positive quadrant.

\begin{Theorem}
  \label{thm:toric-Reeb2}
  Assume that $\p\Omega$ is smooth and convex or concave, or real
  analytic. Then, for a suitable choice of auxiliary structures, the
  number of generators of the symplectic homology complex in the sense
  of Section \ref{sec:SH-def} with action less than or equal to $s$
  is bounded from above by $ C_n(\Omega) s^n + C_0(\Omega)$, where
  the constants $C_n(\Omega)$ and $C_0(\Omega)$ are independent of the
  auxiliary structures and $s$.
\end{Theorem}

The proof of the theorem follows the same path as the proof of Theorem
\ref{thm:toric-sympl2}. Moreover, the argument for convex/concave
domains is implicitly contained in \cite[Sect.\ 2.2 and 2.3]{GH}, and
we omit some of the calculations in that case.

\begin{proof}
  As in the proof of Theorem \ref{thm:toric-sympl2}, we will first
  consider the convex case; the proof for a concave domain $\Omega$ is
  similar. We will view the domain $\Omega=\Omega_f$ as defined by
  \eqref{eq:Omega}. Let $\tOmega:=\Omega_{\tf}$ be a strictly convex,
  smooth $C^\infty$-small perturbation of $\Omega$ with $f$ replaced
  by a smooth function $\tf$ $C^\infty$-close to $f$.

  Denote by $\Delta\subset S^{n-1}$ the intersection of the open
  positive quadrant in a coordinate subspace $V_\Delta$ of any
  dimension $d$ with $S^{n-1}$.  In other words, $\Delta$ is an open
  face of $\bar{\Delta}_+$.  The intersections of $\p \Omega$ and
  $\p \tOmega$ with $V_\Delta$ are parametrized by $f$ and $\tf$ or,
  to be more precise, by the maps
  $\Psi\colon \theta\mapsto (f(\theta),\theta)$ and
  $\tPsi\colon \theta\mapsto (\tf(\theta),\theta)$. (Here we are using
  polar coordinates as in Section \ref{sec:toric}.) Set
  $G_\Delta\colon \bar{\Delta} \to S^{d-1}_\Delta\subset V_\Delta$ to
  be the Gauss map for $\tOmega\cap V_\Delta$, using the outer
  normal, i.e., such that $\theta$ and $G_\Delta(\theta)$ form an acute
  angle. (When the role of $\Delta$ is not essential we will write
  $S^{d-1}$ for the unit sphere $S^{d-1}_\Delta$ in $V_\Delta$.)

  As  a simple calculation shows 
  \begin{equation}
    \label{eq:G}
  G_\Delta=\frac{\p/\p r-P_\Delta \nabla \tf/\tf}{\sqrt{1+\|P_\Delta\nabla \tf\|^2/\tf ^2}},
\end{equation}
where $P_\Delta$ is the orthogonal projection to $V_\Delta$, i.e.,
$P_\Delta\nabla \tf$ is the gradient of the restriction of $\tf$ to
$S^{d-1}_\Delta$ in $V_\Delta$. Furthermore and more importantly, in
\eqref{eq:G} we treat $\nabla \tf$ as the gradient of $\tf$ as a
function on the unit sphere, but not on $\R^n$. Thus, in this equation
$\nabla \tf(\theta)$ is a vector in $T_\theta S^{n-1}$ parallel
transported to the point $(\tf(\theta),\theta)\in V_\Delta$. If we
thought of $\tf$ as a function on $\R^n$ (with the origin deleted) the
gradient at $(r,\theta)$ would be $\nabla \tf/r$. This amounts to the
factor of $1/\tf$ in \eqref{eq:G}.

  The map $G_\Delta$ is a smooth embedding due to the strict convexity
  requirement. In particular, $G_\Delta$ is one-to-one, and
  $G_\Delta(\Delta)$ is an open subset of $S^{d-1}$.

  Note that here again $V_\Delta$ carries a canonical lattice $\Z^d$
  and a canonical measure as the dual of the Lie algebra of $\T^d$.
  Let us call a point $v\in S^{d-1}$ rational if $v$ is the
  intersection of $S^{d-1}$ with a line generated by an integer
  vector.

  The inverse image $L_\theta:=\mu^{-1}(\tPsi(\theta))$, where
  $\theta\in\Delta$, is a torus of dimension $d$ invariant under the
  Reeb flow and also an orbit of the $\T^n$-action. It is easy to see
  that this torus is comprised entirely of periodic orbits if and only
  if $v=G_\Delta(\theta)$ is a rational point in $S^{d-1}$. We will
  call such tori rational. (Otherwise, no point in $L_\theta$ is
  periodic.)  A direct calculation shows that the Reeb action (i.e.,
  period) of an uniterated closed orbit in a rational torus $L_\theta$
  is
  $$
  T(\theta)=\tf(\theta)\left<\theta, G_\Delta(\theta)\right>;
  $$
  cf., \cite[p.\ 3555]{GH}. Explicitly, with \eqref{eq:G} in mind,
  \begin{equation}
    \label{eq:T}
  T(\theta)=\frac{\tf(\theta)}{\sqrt{1+\|P_\Delta\nabla \tf(\theta)\|^2/\tf(\theta) ^2}}.
\end{equation}

Let now $H$ be a semi-admissible Hamiltonian for $\tf$ with slope
$s$. Without loss of generality we may assume that $s$ is not in the
action spectrum of $\mu^{-1}(\p \tOmega)$. A rational point
$v\in G_\Delta(\Delta)$ gives rise to several invariant tori filled up
by one-periodic orbits of $H$ and projecting to $L_\theta$, where
$G_\Delta(\theta)=v$. The number of these tori is
$\lfloor s/T(\theta)\rfloor$.

  Set
  $$
  m_\Delta=\inf \big\{T(\theta)\,\mid \theta\in
  \Delta\big\}.
  $$
  Here the infimum is taken over all $\theta\in \Delta$ but not only
  such that $G_\Delta(\theta)$ is rational. By \eqref{eq:T},
  $$
  m_\Delta=\inf_\Delta \frac{\tf}{\sqrt{1+\|P_\Delta\nabla \tf\|^2/\tf
      ^2}}.
  $$
  It is not hard to see that $m_\Delta>0$ as a consequence only of the
  assumption that $\tf$ is smooth, i.e., $\Omega$ is strictly
  star shaped. For instance, 
  \begin{equation}
    \label{eq:m-f}
  m_\Delta\geq \min_{\bar{\Delta}_+} \frac{\tf}{\sqrt{1+\|\nabla \tf\|^2/\tf
      ^2}}>0.
\end{equation}
Therefore, for each $\Delta$ the total number of such tori is bounded
from above by the number of integer points in the ball of radius
$s/m_\Delta$ in $V_\Delta$. The later number is in turn bounded from
above by $c(d)(s/m_\Delta)^d+c_\Delta$, where as $c(d)$ we can take
any constant greater than the volume of the unit ball and $c_\Delta$
is completely determined by $m_\Delta$ and $c(d)$. Setting
  \begin{equation}
    \label{eq:m}
  m:=\min_{\{\Delta\}} m_\Delta\geq m_{\Delta_+}>0,
\end{equation}
and adding up these upper bounds, we conclude that the number of
rational tori for all $\Delta$ is bounded from above by
  $$
  (2^n-1)C m^{-n} s^n+ C'',
  $$
  where $C=\max\{c(1),\ldots, c(n)\}$ is independent of $\tOmega$. The
  constant $C''$ is completely determined by the constants $c_\Delta$
  and $m_\Delta$ and $C$. We set $C'=(2^n-1)C m^{-n}$.

  Clearly $m_\Delta$ depends $C^1$-continuously on $\tf$, and hence so
  does $m$. As a consequence, the constants $c_\Delta$ and $C''$ can
  also be picked $C^1$-continuously in $\tf$, as is $C'$. Therefore,
  we can replace everywhere $\tOmega$ by $\Omega$ at the expense of
  slightly enlarging the constants $C'$ and $C''$ if necessary. The
  resulting new constants are completely determined by the domain
  $\Omega$.

  Finally, every periodic invariant torus of $H$ is Morse--Bott
  non-degenerate due to the strict convexity requirement for
  $\tOmega$; see again, e.g., \cite[p.\ 3555]{GH}. Hence, under a
  small non-degenerate perturbation of $H$, it would give rise to
  $2^{d}$ non-degenerate fixed points. The interior of the domain
  would in addition give rise to one fixed point. Hence, we obtain
  the upper bound $ C_n(\Omega) s^n + C_0(\Omega)$ on the number of
  generators.

  \begin{Remark}
    \label{rmk:m}
    It is easy to see from the above argument that the constants
    $C_n(\Omega)$ and $C_0(\omega)$ in the upper bound on the number
    of generators in Theorem \ref{thm:toric-Reeb2} can be taken to be
    completely determined by $m$, provided that $\Omega$ is convex or
    concave and $f$ is smooth. To be more precise, without trying to
    minimize these coefficients, we can set
    $C_n(\Omega)= 2^n(2^n-1)C m^{-n}$, where $C=\max c(d)$ is as above
    and $C_0(\Omega)$ is completely determined by $C$ and
    $m$. Moreover, we could replace here $m$ by any constant in the
    range $(0,m]$. These facts play a central role in the proof of
    Theorem~\ref{sec:pf-toric-nonsmooth}.
  \end{Remark}

  Next assume that $\p \Omega$ is real analytic, i.e., by definition,
  the function $f\colon \Delta_+\to (0,\infty) $ parametrizing
  $\p \Omega$ is real analytic and extends real analytically to an
  open neighborhood of $\bar{\Delta}_+$. We keep the notation $f$ for
  such an extension. As in the proof of Theorem
  \ref{sec:pf-toric-sympl}, the convexity of $\p \Omega$ was used
  above only to ensure that $G_\Delta$ is one-to-one, and hence $L_\theta$
  is a single torus, and that this map does not have rational critical
  values, and thus $L_\theta$ is Morse--Bott non-degenerate. Keeping the
  notation from the convex case, we see that it is enough to find an
  arbitrarily $C^2$-small perturbation $\tf$ of $f$ such that

    \begin{itemize}

   \item no rational point $v\in S^{d-1}_\Delta$ is a critical value of
  the Gauss map $\tG_\Delta$;  
  
   \item for any $\Delta$ and any regular value $v\in S^{d-1}_\Delta$ the
  number of inverse images $\tG^{-1}_\Delta(v)$ is bounded from
  above by some constant $N$ completely determined by $\Omega$.

  \end{itemize}

  Indeed, once these requirements are met we can simply set
  $C_n(\Omega)=N2^nC'$ and $C_0(\Omega)=N2^nC''$ where $C'$ and $C''$
  are as above.

  Let $B$ be a sufficiently small closed neighborhood of the unit
  $I\in \SO(n)$. By using the extension of $f$ let us extend $\Omega$ to
  an open (except for the origin) star shaped domain with real
  analytic boundary containing the original domain as a compact
  subset.  We keep the notation $\Omega$ for this extension.  Set
  $\Omega_\lambda:=\lambda (\Omega)\cap \R^{n}_{\geq 0}$ with
  $\lambda\in B$. In other words, $\p \Omega_\lambda$ is parametrized
  by
  $f_\lambda:=f\circ \lambda^{-1}\colon \bar{\Delta}_+\to (0,\infty)$.
  Then the resulting Gauss map
  $G_{\Delta,\lambda}\colon \bar{\Delta}\to S^{d-1}_\Delta$ is the
  unit normal to $\lambda(\p \Omega)\cap V_\Delta$ in $V_\Delta$.  In
  particular, $\Omega_I$ is the original domain, and $\p\Omega_\lambda$ and
  $G_{\Delta,\lambda}$ are $C^\infty$-small perturbations of
  $\p\Omega_I$ and, respectively, $G_\Delta$ when $\lambda$ is close to
  $I$. For every $\Delta$ the function $G_{\Delta,\lambda}(\theta)$ is real
  analytic in $\theta$ and $\lambda$. Therefore, the second
  requirement is satisfied automatically for any $\tf=f_\lambda$,
  i.e., the number of inverse images $\tG^{-1}_{\Delta,\lambda}(v)$ of
  a regular value is bounded from above by a constant $N$ independent
  of $v\in S^{d-1}$ and $\lambda\in B$, and hence completely
  determined by $\Omega$. (See the proof of the real analytic case of
  Theorem \ref{sec:pf-toric-sympl}.)

To complete the proof, we will show that the first requirement is
satisfied when $\tf=f_\lambda$ for $\lambda$ in a second Baire category
set.

Denote by $\Gamma_\Delta(R)$ the set of rational points in
$S^{d-1}_\Delta$ which are the intersections with the lines generated
by the integer lattice vectors in the ball of radius $R$ in
$V_\Delta$. This is a finite set of cardinality $O(R^d)$. Let $U_R$ be
the set of $\lambda\in B$ such that for every $\Delta$ all points in
$\Gamma_\Delta(R)$ are regular values of $G_{\Delta,\lambda}$.  The
set $U_R$ is open since the domain of $G_{\Delta,\lambda}$ is the
closure $\bar{\Delta}$. Hence, we only need to show that $U_R$ is
dense.

To this end let us order the faces $\Delta$ of the simplex
$\bar{\Delta}_+$ in an arbitrary way as
$\Delta_1,\ldots,\Delta_{\ell}$ with $\ell=2^{n-1}-1$. Pick
$\lambda_0\in B$. Let $\rho_1$ be a rotation of $V_{\Delta_1}$,
extended to $\R^n$ as the identity on $V_{\Delta_1}^\perp$, such that
for $\lambda_1=\rho_1\lambda_0$ all points in $\Gamma_{\Delta_1}(R)$
are regular values of $G_{\Delta_1,\lambda_1}$. Clearly, such $\rho_1$
can be taken arbitrarily close to the identity and hence $\lambda_1$
can be made arbitrarily close to $\lambda_0$. Furthermore, this is an
open condition on $\rho_1$ and therefore on $\lambda_1$: for any
$\lambda_2$ close to $\lambda_1$ all points in $\Gamma_{\Delta_1}(R)$
are regular values of $G_{\Delta_1,\lambda_2}$. Next, let $\rho_2$ be
a rotation of $V_{\Delta_2}$, again extended to $\R^n$ as the identity on
$V_{\Delta_2}^\perp$, such that for $\lambda_2=\rho_2\lambda_1$ all
points in $\Gamma_{\Delta_2}(R)$ are regular values of
$G_{\Delta_2,\lambda_2}$. Moreover, we take $\rho_2$ so small that all
points in $\Gamma_{\Delta_1}(R)$ are still regular values of
$G_{\Delta_1,\lambda_2}$.

Proceeding inductively, we obtain an arbitrarily small perturbation
$\lambda_\ell=\rho_\ell\cdots\rho_1\lambda_0$ of $\lambda_0$ such that
for all $\Delta_i$ every point in $\Gamma_{\Delta_i}(R)$ is a regular
value of $G_{\Delta_i,\lambda_\ell}$, i.e., $\lambda_\ell\in
U_R$. Hence, $U_R$ is dense, which completes the proof of the theorem.
\end{proof}

\subsection{Proof of Theorem \ref {thm:toric-nonsmooth}}
\label{sec:pf-toric-nonsmooth}
In the proof we focus first on convex domains and later
indicate the changes needed in the concave case.

The strategy of the proof is to approximate $f$ by smooth strictly
convex functions $\tf$ so that the constant $m$ given by \eqref{eq:m}
for $\tf$ is bounded away from 0 with a lower bound independent of
$\tf$. Then the theorem will follow from the definition of $\fb_\eps$
(see \eqref{eq:fb-C0} and \eqref{eq:fb-C01}) and Remark \ref{rmk:m}.

As in Section \ref{sec:toric}, let us extend $f$ to a small
neighborhood $U$ of $\bar{\Delta}_+$ in $S^{n-1}$ so that the set
$$  
\{(r,\theta)\mid r\leq f(\theta), \,\theta\in
U\}\subset \R^n
$$
given by \eqref{eq:convex-ext} is convex, where we keep the notation
$f$ for the extension. Throughout the proof we will shrink the
neighborhood $U$ several times.

Consider the function $F(x)=r/f(\theta)$, where we are using polar
coordinates $x=(r,\theta)$ on $\R^{n}$. This function is defined on
the cone spanned by $f(U)$, homogeneous of degree one and convex. As a
consequence, $F$ is uniformly Lipschitz on the cone spanned, perhaps,
by a smaller neighborhood of $\bar{\Delta}_+$, still denoted by $U$;
see \cite{Ro}. (This is the first instance when we need to shrink
$U$.)  Therefore, $f$ is also uniformly Lipschitz on $U$ after if
unnecessary shrinking $U$ again.  Denote by $\CC$ the convex set left
after cutting off the tip of the cone by an affine anti-diagonal
hyperplane, i.e., obtained by removing from this cone its intersection
with the half-space $\sum y_i\leq a $ for a sufficiently small
constant $a>0$. (Here $y_1,\ldots,y_n$ are the coordinates on $\R^n$.)

Let us approximate $F$ by smooth convex functions by applying the
following standard procedure; cf.\ \cite[Sec.\ 1]{Az}. Fix a
non-negative, compactly supported function $\delta$ on $\R^n$ with
integral equal to 1. Then the family
$\delta_\eta(x)=\eta^{-n}\delta(x/\eta)$, $\eta>0$, is an approximate
identity. When $\eta$ is small, the convolutions
$$
F_\eta=F*\delta_\eta
$$
are defined on the cut-off cone obtained from $\CC$ by slightly
shrinking $U$, which we still denote by $\CC$.  These functions are
smooth, convex, and globally Lipschitz with the same Lipschitz
constant as $F$. Moreover, $F_\eta$ are uniformly $C^0$-close to $F$
when $\eta$ is small.  In the notation from Section \ref{sec:toric},
set
$$
\Sigma=\{x\mid F(x)=1\}=:\p \Omega_f\textrm{ and }
\Sigma_\eta=\{x\mid F_\eta(x)=1\}\subset \CC.
$$
We claim that $1$ is a regular value of $F_\eta$ and, as a consequence,
$\Sigma_\eta$ is a smooth hypersurface. In fact, we will prove a
stronger result which is essential for our purposes: the derivative of
$F_\eta$ in the radial direction at the points of $\Sigma_\eta$ is
bounded from below by a constant $\xi>0$ which depends only on $f$. To
be more precise, we have the following.

\begin{Lemma}
  \label{lemma:dir}
There exists a constant $\xi>0$ depending only on $f$ such that
whenever $\eta>0$ is small,
$$
F_\eta(r_1,\theta)-F_\eta(r_0,\theta)\geq \xi(r_1-r_0)
$$
for every $\theta$ in $U$ and any $r_1>r_0$ sufficiently close to
$f(\theta)$ uniformly on $U$.
\end{Lemma}

As an immediate consequence of the lemma, we obtain the lower bound
\begin{equation}
  \label{eq:derivative}
\frac{\p F_\eta}{\p r}\geq \xi \textrm{ along } \Sigma_\eta.
\end{equation}
For instance, it is easy to see that if $f$ were smooth, the left hand
side of \eqref{eq:derivative} would be close to $\p F/ \p r=1/f$ and
for a small $\eta>0$ we could take as $\xi$ any constant below
$\inf (1/f)$.  We will show that in fact this is also true when $f$
is just continuous.

\begin{proof}[Proof of Lemma \ref{lemma:dir}]
  Set $x_1=(r_1,\theta)$ and $x_0=(r_0,\theta)$, and let $y$ be in
  $Y_\eta:=\supp \delta_\eta$. Define the functions $r_i(y)$ and
  $\theta_i(y)$ for $i=0,\, 1$ by $x_i-y=(r_i(y),\theta_i(y))$ in
  polar coordinates. Then
\begin{equation*}
  \begin{split}
    F_\eta(x_1)-F_\eta(x_0)
     &=\int\big( F(x_1-y)-F(x_0-y)\big) \delta_\eta(y)\, dy\\
    & \geq \int\big( F(r_1(y),\theta_0(y))- F(r_0(y),\theta_0(y))\big)
    \delta_\eta(y)\, dy\\
    & \quad - \int\big| F(r_1(y),\theta_1(y))- F(r_1(y),\theta_0(y))\big| \delta_\eta(y)\, dy.
  \end{split}
\end{equation*}
Our goal is to bound the first term from below and the second one term
from above. For the first term, we have
$$
\int\big( F(r_1(y),\theta_0(y))- F(r_0(y),\theta_0(y))\big)
\delta_\eta(y)\, dy\geq \xi_0(r_1-r_0)
$$
for some constant $\xi_0>0$ which depends only on $f$. For instance, 
we can take $\xi_0=\inf_U (1/f)$ after shrinking $U$ if necessary and
requiring $\eta>0$ to be sufficiently small.

To bound the second term from above, first note that 
\begin{equation*}
  \begin{split}
    \int\big| F(r_1(y),\theta_1(y))- F(r_1(y),\theta_0(y))\big|
    \delta_\eta(y)\, dy &=\int \left| \frac{r_1(y)}{f(\theta_1(y))}-
      \frac{r_1(y)}{f(\theta_0(y))}\right|\delta_\eta(y)\, dy\\
    &\leq C \cdot\sup_{Y_\eta}\big| f(\theta_1(y))-f(\theta_0(y))\big|
  \end{split}
\end{equation*}
since
$F(r,\theta)=r/f(\theta)$, where $Y_\eta=\supp \delta_\eta$. Here the constant $C$ is completely
determined by $f$ provided that $r_0$ and $r_1$ are sufficiently close
to $f(\theta)$ and $\eta>0$ is small. Denote by $\rho$ the distance in
$S^{n-1}$. A routine calculation shows that under the above
conditions
$$
d(\theta_1(y),\theta_0(y))\leq O(\|y\|)|r_1-r_0|.
$$
Since $f$ is Lipschitz, the second term is bounded from above by
$O(\eta)|r_1-r_0|$. Hence, we can take as $\xi$ any constant smaller
than $\xi_0$ whenever $r_0$ and $r_1$ are sufficiently close to
$f(\theta)$ uniformly in $\theta$ and $\eta>0$ is small. This
completes the proof of the lemma.
\end{proof}

By \eqref{eq:derivative}, $\Sigma_\eta$ is a smooth convex
hypersurface, lying in a small neighborhood of $\Sigma$. Moreover,
$\Sigma_\eta$ intersects every line through the origin at most
once since the restriction of
$F_\eta$ to a line is convex and $F_\eta(r,\theta)$ is close to
$F(0)=0$ when $r$ is close to zero.  Hence,
while $\Sigma$ is given by the equation $r=f(\theta)$, the
hypersurface $\Sigma_\eta$ is given by the equation $r=f_\eta(\theta)$
for some smooth function $f_\eta$ on $U$, which is $C^0$-close to
$f$. This follows from the inverse function theorem. In addition, the
second fundamental form of $\Sigma_\eta$ is negative semi-definite,
due to the fact $F_\eta$ is smooth and convex.

The gradient of $f_\eta$ is bounded from above
by a constant completely determined by $f$, provided again that
$\eta>0$ is small. Namely,
$$
\|\nabla f_\eta\|\leq L/\xi,
$$
where $L$ is the Lipschitz constant of $F$ and hence also of
$F_\eta$. To see this, we simply differentiate the identity
$F_\eta(f_\eta(\theta),\theta)=1$ with respect to $\theta$ and use
\eqref{eq:derivative}.

Our next goal is to show that $m$ from \eqref{eq:m} for $f_\eta$ is
bounded away from 0 by a constant completely determined by $f$
provided that $\eta>0$ is sufficiently small. By Remark \ref{rmk:m}
this will imply that the number of generators in the symplectic
homology complex for $f_\eta$ is bounded from above by
$C_n s^n+C_0$ with $C_n$ and $C_0$ independent of $\eta$ and
completely determined by $f$. Then the theorem will follow from the
definition of $\fb_\eps(s)$ for $W_f$; see \eqref{eq:fb-C0}
and \eqref{eq:fb-C01}. 

By \eqref{eq:m-f} and \eqref{eq:m} applied to $\tf=f_\eta$ we have for
$\eta>0$ sufficiently small
$$
m:=\min m_\Delta\geq \min_{\bar{\Delta}_+}
\frac{f_\eta}{\sqrt{1+\|\nabla f_\eta\|^2/f_\eta ^2}}\geq \frac{\min
  f}{2\sqrt{1+L^2/(\xi\max f)^2}}
$$
with the ultimate lower bound independent of the smooth approximation
$f_\eta$. As in Remark \ref{rmk:m}, replacing $m$ by this lower bound
we obtain an upper bound $ C_n(\Omega) s^n + C_0(\Omega)$ on the
growth rate of the number of generators, and hence $\fb_\eps(s)$,
with constants $C_n(\Omega)$ and $C_0(\Omega)$ completely determined
by $\Omega=\Omega_f$ and independent of $s$ and $\eps>0$.

By construction, the hypersurface $\Sigma_\eta$ is convex, but not
necessarily strictly convex. However, it can be approximated by
strictly convex hypersurfaces without changing the lower bound on $m$
or just decreasing it by an arbitrarily small amount. To this end, we can
simply add a sufficiently small positive definite quadratic form to
$F_\eta$. This concludes the proof of Theorem
\ref{sec:pf-toric-nonsmooth}.

The argument goes through for concave domains with straightforward
modifications in the setup. Namely, as above we extend $f$ to a small
neighborhood $U$ of $\bar{\Delta}_+$. Then the function $F$ is
replaced by $-F$ and the truncated cone $\CC$ is replaced by the
intersection of the set
$\{(r,\theta)\mid r\geq f(\theta),\, \theta\in U\}$ with a sufficiently
large ball. \qed

\end{document}